\documentclass{article}
\usepackage{graphicx,verbatim, bussproofs,tikz,float, latexsym}
\usetikzlibrary{arrows}
\usepackage[left=3cm,top=3cm,right=3cm,bottom=2cm]{geometry}
\usepackage{pgfplots,multicol}
\usepackage{amsmath,amsthm}
\usepackage{amssymb}
\usepackage{stmaryrd}
\usepackage{proof}
\usepackage{cmll}
\usepackage{bussproofs}
\usepackage{authblk}
\DeclareSymbolFont{extraup}{U}{zavm}{m}{n}
\DeclareMathSymbol{\vardiamond}{\mathalpha}{extraup}{87}

\newcommand{\commment}[1]{}

\newcommand{\nomi}{\mathbf{i}}
\newcommand{\nomj}{\mathbf{j}}

\newcommand{\bigamp}{\mathop{\mbox{\Large \&}}}

\renewcommand{\phi}{\varphi}
\renewcommand{\emptyset}{\varnothing}

\newcommand{\ca}{\mathbb{A}^\delta}
\newcommand{\A}{\mathbb{A}}
\newcommand{\B}{\mathbb{B}}

\newcommand{\bbas}{\mathbb{A}^{\delta}}

\newcommand{\Boxblack}{\blacksquare\,}
\newcommand{\Diamondblack}{\vardiamond}

\renewcommand{\epsilon}{\varepsilon}
\newtheorem{theorem}{Theorem}[section]
\newtheorem{lemma}[theorem]{Lemma}
\newtheorem{prop}[theorem]{Proposition}

\newtheorem{cor}[theorem]{Corollary}

\newtheorem{example}[theorem]{Example}

\theoremstyle{definition}
\newtheorem{definition}[theorem]{Definition}
\newtheorem{remark}[theorem]{Remark}

\title{Algorithmic Correspondence and Canonicity for Possibility Semantics}
\author{Zhiguang Zhao}
\date{}
\begin{document}
\maketitle
\begin{abstract}
The present paper develops a unified correspondence treatment of the Sahlqvist theory for possibility semantics, extending the results in \cite{Ya16} from Sahlqvist formulas to the strictly larger class of inductive formulas, and from the full possibility frames to filter-descriptive possibility frames.
Specifically, we define the possibility semantics version of the algorithm ALBA, and an adapted interpretation of the expanded modal language used in the algorithm. We prove the soundness of the algorithm with respect to both (the dual algebras of) full possibility frames and (the dual algebras of) filter-descriptive possibility frames. We make some comparisons among different semantic settings in the design of the algorithms, and fit possibility semantics into this broader picture.

One notable feature of the adaptation of ALBA to possibility frames setting is that the so-called nominal variables, which are interpreted as complete join-irreducibles in the standard setting, are interpreted as regular open closures of ``singletons'' in the present setting.

\end{abstract}

\section{Introduction}

\paragraph{Possibility semantics.} Possibility semantics was proposed by Humberstone \cite{Hu81} as an alternative semantics for modal logic, which is based on possibilities rather than possible worlds in Kripke semantics, where every possibility does not provide truth values of all propositions, but only some of them. Different possibilities are ordered by a refinement relation where some possibilities provide more information about the truth value of propositions than others. 

In recent years, possibility semantics has been intensely investigated: \cite{Ho14} gives a construction of canonical possibility models with a finitary flavor, \cite{BeHo16} studies the intuitionistic generalization of possibility models, \cite{HT16a} investigates the first-order counterpart of possibility semantics, \cite{HT16b} focuses on the relation between Kripke models and possibility models, \cite{vBBeHo16} provides a bimodal perspective on possibility semantics. A comprehensive study of possibility semantics can be found in \cite{Ho16}. In \cite{Ya16}, the first author of the present paper investigates the correspondence theory for possibility semantics and proves a correspondence theorem for Sahlqvist formulas over full possibility frames---which are the counterparts of (full) Kripke frames in the possibility semantics setting---using insights from algebraic correspondence theory developed in \cite{ConPalSou}. In \cite[Theorem 7.20]{Ho16}, Holliday shows that all {\em inductive} formulas are filter-canonical (i.e. their validity is preserved from the canonical possibility models to their underlying canonical full possibility frames), hence every normal modal logic axiomatized by inductive formulas is sound and complete with respect to its canonical full possibility frame. Inductive formulas \cite{GorankoV06} form a syntactically defined class of formulas which properly extends Sahlqvist formulas, while enjoying their same property of canonicity and correspondence \cite{CoGoVa06}. Holliday's result provides the canonicity half of the Sahlqvist-type result for inductive formulas relative to possibility semantics. The present paper provides the remaining half. Namely, we show that inductive formulas have first-order correspondents in full possibility frames as well as in filter-descriptive possibility frames\footnote{As remarked in \cite[page 103]{Ho16}, correspondence results might be lost when moving from full possibility frames to filter-descriptive possibility frames.}. 

\paragraph{Unified correspondence.} 

The methodology of the present paper pertains to a wider theory, referred to as \emph{unified correspondence}. As explained in \cite{ConPalSou,CoGhPa14}, this theory is built on the duality between the algebraic and the relational semantics of non-classical logics, thanks to which it has explained the ``Sahlqvist phenomenon'' in terms of the algebraic and order-theoretic properties of the interpretations of the logical connectives. The focus on these properties has been key to being able to generalize the Sahlqvist-type results from modal logic to a wide array of logics, which include, among others, intuitionistic and distributive and general (non-distributive) lattice-based (modal) logics \cite{CoPa12,CoPa:non-dist,CFPPTW}, non-normal (regular) modal logics based on distributive lattices of arbitrary modal signature \cite{PaSoZh16}, monotone modal logic \cite{FrPaSa16}, hybrid logics \cite{ConRob}, many valued logics \cite{LeRoux:MThesis:2016} and bi-intuitionistic and lattice-based modal mu-calculus \cite{CoCr14,CCPZ,CFPS15}. The main tools of this theory are a language-independent definition of {\em inductive formulas/inequalities} based on the order-theoretic properties of the interpretations of the logical connectives, and the algorithm ALBA, which computes the first-order correspondent of input formulas/inequalities and is guaranteed to succeed on the inductive class. Subsequent work has shown that the algorithm ALBA is much more versatile as a tool than initially expected. Indeed, a suitable adaptation of ALBA has been used to compute the shape of the analytic rules equivalent to given axioms so as to contribute to the quest of uniform methods to design analytic calculi for non-classical logics \cite{GMPTZ,MZ16}. Other suitable modifications of the original algorithm have been used to prove canonicity results in settings progressively further away from correspondence: J\'onsson-style canonicity \cite{PaSoZh15}, constructive canonicity \cite{CP:constructive}, and canonicity via pseudo-correspondence \cite{CPSZ}. In particular, in \cite{CP:constructive}, ALBA has been used to prove {\em constructive canonicity}, i.e.\ algebraic canonicity in a context in which relational semantics is not defined in general. Interestingly, in the same paper, the proof technique used to prove the main result has also been systematically related to the methodology used to prove {\em canonicity-via-correspondence}. The contributions of the present paper can be understood as providing yet another dimension to this story: indeed, the possibility semantics can be understood as the relational incarnation of the general algebraic environment in which constructive canonicity is proved. 

\paragraph{Methodology.} We analyze the correspondence phenomenon in possibility semantics using the dual algebraic structures of (full) possibility frames, namely complete (not necessarily atomic) Boolean algebras with complete operator, where atoms are not always available. For correspondence over full possibility frames, we identify two different Boolean algebras with operator as the dual algebraic structures of a given full possibility frame by viewing the full possibility frame in two different ways, namely the Boolean algebra of regular open subsets $\mathbb{B}_{\mathsf{RO}}$ (when viewing the possibility frame as a possibility frame itself) and the Boolean algebra of arbitrary subsets $\mathbb{B}_{\mathsf{K}}$ (when viewing the possibility frame as a bimodal Kripke frame, see \cite{vBBeHo16}), where a canonical order-embedding map $e:\mathbb{B}_{\mathsf{RO}}\to\mathbb{B}_{\mathsf{K}}$ can be defined. The embedding $e$ preserves arbitrary meets, therefore a left adjoint $c:\mathbb{B}_{\mathsf{K}}\to\mathbb{B}_{\mathsf{RO}}$ of $e$ can be defined, which sends a subset $X$ of the domain $W$ of possibilities to the smallest regular open subset containing $X$. This left adjoint $c$ plays an important role in the dual characterization of the interpretations of nominals and the black connectives, which form the ground of the regular open translation of the expanded modal language. In particular, we give an algebraic counterpart of Lemma 3.7 in \cite{Ya16} that every regular open element can be represented as the join of regular open closures of singletons below it, therefore the regular open closures of singletons form the join-generators. When it comes to canonicity, we prove a topological Ackermann lemma similar to \cite[Lemma 9.3 and 9.4]{CoPa12}, which forms the basis of the correspondence result with respect to filter-descriptive frames as well as the canonicity result.

\paragraph{Structure of the paper.} Section \ref{aSec:Prelim} presents preliminaries on possibility semantics, both frame-theoretically and algebraically, as well as the duality theory background of possibility semantics. Section \ref{aSec:algebraic:analysis} gives an algebraic analysis of the semantic environment of possibility semantics for the interpretation of the expanded modal language, the details of which will be given in Section \ref{aSec:expanded:language} together with the regular open translation and the syntactic definition of Sahlqvist and inductive formulas. The Ackermann Lemma Based Algorithm (ALBA) for possibility semantics is given in Section \ref{aSec:ALBA} as well as some examples, with its soundness proof with respect to full possibility frames in Section \ref{aSec:soundness} and the proof that it succeeds on inductive formulas in Section \ref{aSec:success}. The soundness proof with respect to filter-descriptive possibility frames and the canonicity-via-correspondence proof are given in Section \ref{aSec:Canonicity}. Section \ref{aSec:Discussion} provides some discussions, and gives some further directions.

\section{Preliminaries on possibility semantics}\label{aSec:Prelim}

In the present section we collect the preliminaries on possibility semantics. For more details, see e.g.\ \cite[Section 1 and 2]{Ho16} and \cite{Ya16}. 

\subsection{Language}

Given a set $\mathsf{Prop}$ of propositional variables, the basic modal language $\mathcal{L}$ is defined as follows:
$$\varphi::=p \mid \neg\varphi \mid \varphi\land\varphi \mid \Box\varphi,$$
where $p\in \mathsf{Prop}$. We define $\varphi\lor\psi:=\neg(\neg\varphi\land\neg\psi)$, $\varphi\to\psi:=\neg\varphi\lor\psi$, $\bot:=p\land\neg p$, $\top:=\neg\bot$ and $\Diamond\varphi:=\neg\Box\neg\varphi$, respectively. We also use $\mathsf{Prop}(\alpha)$ to denote the propositional variables occuring in $\alpha$. \footnote{In the present paper we will consider only the modal language with only one unary modality. These results can be easily generalized to languages with arbitrary signature, which is ongoing work.}

We will find it convenient to work on \emph{inequalities}, i.e.\ expressions of the form $\phi\leq\psi$, the interpretation of which is equivalent to the implicative formula $\phi\to\psi$ being true at any point in a model. Throughout the paper, we will also make substantial use of \emph{quasi-inequalities}, i.e.\ expressions of the form $\phi_1\leq\psi_1\ \&\ \ldots\ \&\ \phi_n\leq\psi_n\Rightarrow\phi\leq\psi$, where $\&$ is the meta-level conjunction and $\Rightarrow$ is the meta-level implication.

\subsection{Downset topology}
In the present subsection, we report on the definition of possibility frames and possibility models. We will make use of the following auxiliary notions. For every partial order $(W, \sqsubseteq)$, a subset $Y\subseteq W$ is {\em downward closed} (or a {\em down-set}) if for all $x, y\in W$, if $x\in Y$ and $y\sqsubseteq x$, then $y\in Y$.
For every $X\subseteq W$, the set ${\Downarrow}X:=\{x\in W\mid(\exists y\sqsupseteq x)(y\in X)\}$ is the smallest down-set containing $X$. The set of all down-sets of $(W, \sqsubseteq)$ forms a topology on $W$, denoted by $\tau_{\sqsubseteq}$, which we call \emph{the downset topology}.

For any $X\subseteq W$, we let $\mathsf{cl}(X):=\{x\in W\mid (\exists y\sqsubseteq x)(y\in X)\}$ (resp.\ $\mathsf{int}(X):=\{x\in W\mid (\forall y\sqsubseteq x)(y\in X)\}$) denote the \emph{closure} (resp.\ \emph{interior}) of $X$. We also let $$\mathsf{RO}(W,\tau_{\sqsubseteq}):=\{X\subseteq W\mid \mathsf{int}(\mathsf{cl}(X))=X\}$$ denote the collection of \emph{regular open} subsets of $W$. We say a set $Y\subseteq W$ the \emph{regular open closure} of $X$ if $Y$ is the least regular open subset of $W$ containing $X$, and denote $Y=\mathsf{ro}(X)$.

We collect some useful facts about the downset topology:

\begin{prop}(cf.\ \cite[page 16-18]{Ho16})\label{afacts:regular:open}
For every partial order $(W, \sqsubseteq)$,
\begin{enumerate}
\item every regular open subset of $(W, \tau_{\sqsubseteq})$ is a down-set, and hence a $\tau_{\sqsubseteq}$-open subset.
\item $\mathsf{ro}(X)=\mathsf{int}(\mathsf{cl}(\Downarrow X))$ for any subset $X\subseteq W$.
\item $\mathsf{ro}(X)=\mathsf{int}(\mathsf{cl}(X))$ for any $X\in\tau_{\sqsubseteq}$.
\item $\emptyset, W\in\mathsf{RO}(W,\tau_{\sqsubseteq})$.
\item $X\cap Y\in\mathsf{RO}(W,\tau_{\sqsubseteq})$ if $X, Y\in\mathsf{RO}(W,\tau_{\sqsubseteq})$.
\item $\mathsf{int}(\mathsf{cl}(X\cup Y))\in\mathsf{RO}(W,\tau_{\sqsubseteq})$ if $X, Y\in\mathsf{RO}(W,\tau_{\sqsubseteq})$.
\item $\mathsf{int}(W\setminus X)\in\mathsf{RO}(W,\tau_{\sqsubseteq})$ if $X\in\mathsf{RO}(W,\tau_{\sqsubseteq})$.
\item $\mathsf{RO}(W,\tau_{\sqsubseteq})$ is closed under arbitrary intersection (cf.\ \cite[footnote 13 on page 17]{Ho16}).
\item $(\mathsf{RO}(W,\tau_{\sqsubseteq}), \emptyset,W,\land,\lor,-)$ is a Boolean algebra such that for all $X, Y\in(\mathsf{RO}(W,\tau_{\sqsubseteq})$, 
$$X\land Y=X\cap Y\qquad X\lor Y=\mathsf{int}(\mathsf{cl}(X\cup Y))\qquad-X=\mathsf{int}(W\setminus X).$$
\item for all $X, Y\in(\mathsf{RO}(W,\tau_{\sqsubseteq})$, 

$$X\supset Y:=-X\lor Y=\mathsf{int}((W\setminus X)\cup Y)$$
\end{enumerate}
\end{prop}

\subsection{Relational semantics}

For every binary relation $R$ on a set $W$, we denote $R[X]=\{w\in W\mid (\exists x\in X)Rxw\}$ and $R^{-1}[X]=\{w\in W\mid (\exists x\in X)Rwx\}$, and denote $R[w]:=R[\{w\}]$ and $R^{-1}[w]:=R^{-1}[\{w\}]$, respectively.
Below we give a slightly different but equivalent definition of possibility frames than the one given in \cite{Ho16}.

\begin{definition}[Possibility frames and models]\label{adef:poss:frame:model}

A \emph{possibility frame} is a tuple $\mathbb{F}=(W, \sqsubseteq, R, \mathsf{P})$, where $W\neq\emptyset$ is the \emph{domain} of $\mathbb{F}$, the \emph{refinement relation}\footnote{We adopt the order of the refinement relation as in \cite{Ho16,Ya16}, which is used in the theory of weak forcing \cite{Od89}, while in the literature of intuitionistic logic, the order is typically the reverse order.} $\sqsubseteq$ is a partial order on $W$, the \emph{accessibility relation} $R$ is a binary relation\footnote{In the literature, some interaction conditions are imposed between the accessibility relation and the refinement relation (cf.\ \cite{vBBeHo16}). Since these conditions are not needed for our treatment, we do not impose them and we do not discuss them any further in the present paper.} on $W$, and the collection $\mathsf{P}\subseteq\mathsf{RO}(W,\tau_{\sqsubseteq})$ of \emph{admissible subsets} forms a sub-Boolean algebra of $\mathsf{RO}(W,\tau_{\sqsubseteq})$ such that $\Box_{\mathsf{P}}(X)=\{w\in W\mid R[w]\subseteq X\}\in\mathsf{P}$ for any $X\in\mathsf{P}$. A \emph{pointed possibility frame} is a pair $(\mathbb{F}, w)$ where $w\in W$.
A \emph{possibility model} is a pair $\mathbb{M}=(\mathbb{F}, V)$ where $V:\mathsf{Prop}\to\mathsf{P}$ is a \emph{valuation} on $\mathbb{F}$.
A possibility frame is \emph{full} if $\mathsf{P}=\mathsf{RO}(W,\tau_{\sqsubseteq})$. 
\end{definition}

It follows straightforwardly from the conditions above that $\mathsf{P}$ is endowed with the algebraic structure of a BAO.\label{aP:vs:RO} In general, for any possibility frame $\mathbb{F}=(W, \sqsubseteq, R, \mathsf{P})$, the \emph{underlying full possibility frame} $\mathbb{F}^{\sharp}=(W, \sqsubseteq, R, \mathsf{RO}(W,\tau_{\sqsubseteq}))$ might not be well-defined, since $\mathsf{RO}(W,\tau_{\sqsubseteq})$ might not be closed under the box operation arising from the relation $R$. However, in certain situations which we will discuss in Remark \ref{aremark:underlying:full}, the underlying full possibility frame is well-defined.

Given any possibility model $\mathbb{M}=(W, \sqsubseteq, R, \mathsf{P}, V)$ and any $w\in W$, the \emph{satisfaction relation} is defined as follows:

\begin{itemize}
\item $\mathbb{F}, V, w\vDash p$ iff $w\in V(p)$;
\item $\mathbb{F}, V, w\vDash \varphi\land\psi$ iff $\mathbb{F}, V, w\vDash \varphi$ and $\mathbb{F}, V, w\vDash\psi$;
\item $\mathbb{F}, V, w\vDash \neg\varphi$ iff $(\forall v\sqsubseteq w)(\mathbb{F}, V, v\nvDash\varphi)$;
\item $\mathbb{F}, V, w\vDash \Box\varphi$ iff $\forall v(Rwv\ \Rightarrow\ \mathbb{F}, V, v\vDash\varphi)$.
\end{itemize}

For any formula $\phi$, we let $\llbracket\varphi\rrbracket^{\mathbb{M}}=\{w\in W\mid \mathbb{M}, w\vDash\varphi\}$ denote the \emph{truth set} of $\varphi$ in $\mathbb{M}$. The formula $\varphi$ is \emph{globally true} on a possibility model $\mathbb{M}$ (notation: $\mathbb{M}\vDash\varphi$) if $\mathbb{M}, w\vDash\varphi$ for every $w\in W$. Moreover, $\varphi$ is \emph{valid} on a pointed possibility frame $(\mathbb{F}, w)$ (notation: $\mathbb{F}, w\vDash\varphi$) if $\mathbb{F}, V, w\vDash\varphi$ for every valuation $V$. We say that $\varphi$ is \emph{valid} on a possibility frame $\mathbb{F}$ (notation: $\mathbb{F}\vDash\varphi$) if $\varphi$ is valid on $(\mathbb{F}, w)$ for every $w\in W$. The definition of global truth for inequalities and quasi-inequalities is given as follows:

\begin{itemize}
\item $\mathbb{F}, V\vDash \varphi\leq\psi$ $\qquad$ iff $\qquad$ for all $w\in W$, if $\mathbb{F}, V, w\vDash\varphi$ then $\mathbb{F}, V, w\vDash\psi$;
\item $\mathbb{F}, V\vDash\bigamp_{i=1}^{n}(\phi_i\leq\psi_i)\Rightarrow\varphi\leq\psi$$\qquad$ iff $\qquad$ if $\mathbb{F}, V\vDash\phi_i\leq\psi_i$ for all $i$ then $\mathbb{F}, V\vDash\varphi\leq\psi$.
\end{itemize}

An inequality (resp.\ quasi-inequality) is valid on $\mathbb{F}$ if it is globally true on $(\mathbb{F}, V)$ for every valuation $V$.

It is easy to check that the following truth conditions hold for defined connectives as well as inequalities:
\begin{prop}\label{aprop:additional:connectives}
The following equivalences hold for any possibility frame $\mathbb{F}=(W, \sqsubseteq, R, \mathsf{P})$, any valuation $V$ and any $w\in W$:
\begin{itemize}
\item $\mathbb{F}, V, w\vDash\top$: always;
\item $\mathbb{F}, V, w\vDash\bot$: never;
\item $\mathbb{F}, V, w\vDash \varphi\lor\psi$ iff $(\forall v\sqsubseteq w)(\exists u\sqsubseteq v)(\mathbb{F}, V, u\vDash \varphi$ or $\mathbb{F}, V, u\vDash\psi$);
\item $\mathbb{F}, V, w\vDash \varphi\to\psi$ iff $(\forall v\sqsubseteq w)(\mathbb{F}, V, v\vDash\varphi\ \Rightarrow\ \mathbb{F}, V, v\vDash\psi)$;
\item $\mathbb{F}, V, w\vDash \Diamond\varphi$ iff $(\forall v\sqsubseteq w)\exists u(Rvu\land(\exists t\sqsubseteq u)(\mathbb{F}, V, t\vDash \varphi))$;
\item $\mathbb{F}, V\vDash \varphi\to\psi$ iff $\llbracket\varphi\rrbracket^{\mathbb{F}, V}\subseteq\llbracket\psi\rrbracket^{\mathbb{F}, V}$ iff $\mathbb{F}, V\vDash \varphi\leq\psi$;
\item $\mathbb{F}\vDash\varphi\to\psi$ iff $\mathbb{F}\vDash\varphi\leq\psi$.
\end{itemize}
\end{prop}

The following proposition can be understood as stating that the interpretation of modal formulas on possibility frames and models can be obtained from the standard algebraic semantics for modal logic via the duality established in \cite{Ho16}:

\begin{prop}
For any possibility frame $\mathbb{F}=(W, \sqsubseteq, R, \mathsf{P})$, any valuation $V:\mathsf{Prop}\to\mathsf{P}$ and any $w\in W$,
\begin{itemize}
\item $\llbracket\top\rrbracket^{\mathbb{F},V}=W$;
\item $\llbracket\bot\rrbracket^{\mathbb{F},V}=\emptyset$;
\item $\llbracket\varphi\land\psi\rrbracket^{\mathbb{F},V}=\llbracket\varphi\rrbracket^{\mathbb{F},V}\cap\llbracket\psi\rrbracket^{\mathbb{F},V}$;
\item $\llbracket\varphi\lor\psi\rrbracket^{\mathbb{F},V}=\mathsf{ro}(\llbracket\varphi\rrbracket^{\mathbb{F},V}\cup\llbracket\psi\rrbracket^{\mathbb{F},V})=\mathsf{int}(\mathsf{cl}(\llbracket\varphi\rrbracket^{\mathbb{F},V}\cup\llbracket\psi\rrbracket^{\mathbb{F},V}))$;
\item $\llbracket\neg\varphi\rrbracket^{\mathbb{F},V}=\mathsf{int}(W\setminus\llbracket\varphi\rrbracket^{\mathbb{F},V})$;
\item $\llbracket\varphi\to\psi\rrbracket^{\mathbb{F},V}=\mathsf{int}((W\setminus\llbracket\varphi\rrbracket^{\mathbb{F},V})\cup\llbracket\psi\rrbracket^{\mathbb{F},V})$;
\item $\llbracket\Box\varphi\rrbracket^{\mathbb{F},V}=W\setminus R^{-1}[W\setminus\llbracket\varphi\rrbracket^{\mathbb{F},V}]$;
\item $\llbracket\Diamond\varphi\rrbracket^{\mathbb{F},V}=\mathsf{int}(R^{-1}[\mathsf{cl}(\llbracket\varphi\rrbracket^{\mathbb{F},V})])$.
\end{itemize}
\end{prop}

\subsection{Algebraic semantics}\label{aSec:Algebraic:Semantics}

Thanks to the duality in \cite{Ho16}, we will be able to work throughout the paper in the environment of the dual algebras of the possibility frames, namely the Boolean algebras with operators.

\begin{definition}[Boolean algebra with operator]
A Boolean algebra with operator (BAO) is a tuple $\mathbb{B}=(B,\bot,\top,\land,\lor,-,\Box)$, where $(B, \bot,\top,\land,\lor,-)$ is a Boolean algebra and moreover, $\Box\top=\top$ and $\Box(a\land b)=\Box a\land \Box b$ for any $a,b\in B$. The order on $\mathbb{B}$ is defined as $a\leq b$ iff $a\land b=a$. We will sometimes abuse notation and use $\mathbb{B}$ to denote $B$.

For any BAO $\mathbb{B}$, let $\mathbb{B}^{\partial}$ denote its order-dual, and let $\mathbb{B}^1 = \mathbb{B}$. For any $n\in N$, an \emph{$n$-order type} $\epsilon$ is an element of $\{1, \partial\}^n$, and its $i$-th coordinate is denoted by $\epsilon_i$. We omit $n$ when it is clear from the context. Let $\epsilon^\partial$ denote the \emph{dual order type} of $\epsilon$, i.e.\ the order-type such that $\epsilon^\partial_i=1$ (resp.\ $\partial$) if $\epsilon_i=\partial$ (resp.\ $1$). For any $n$-order type $\epsilon$, we let $\mathbb{B}^{\epsilon}$ be the product algebra $\mathbb{B}^{\epsilon_1}\times\ldots\times\mathbb{B}^{\epsilon_n}$.

An \emph{assignment} on $\mathbb{B}$ is a map $\theta:\mathsf{Prop}\to\mathbb{B}$, which can be extended to all formulas as usual. We use $\varphi^{(\mathbb{B}, \theta)}$ or $\theta(\varphi)$ to denote the value of $\varphi$ in $\mathbb{B}$ under $\theta$. We say that a formula $\varphi$ (resp.\ an inequality $\varphi\leq\psi$) is \emph{true} on $\mathbb{B}$ under $\theta$ (notation: $(\mathbb{B},\theta)\vDash\varphi$, $(\mathbb{B},\theta)\vDash\varphi\leq\psi$), if $\varphi^{(\mathbb{B}, \theta)}=\top$ (resp.\ $\varphi^{(\mathbb{B}, \theta)}\leq\psi^{(\mathbb{B}, \theta)}$), and $\varphi$ (resp.\ $\varphi\leq\psi$) is \emph{valid} on $\mathbb{B}$ (notation: $\mathbb{B}\vDash\varphi$, $\mathbb{B}\vDash\varphi\leq\psi$) if $(\mathbb{B},\theta)\vDash\varphi$ (resp.\ $(\mathbb{B},\theta)\vDash\varphi\leq\psi$) for every $\theta$. 

A quasi-inequality $\bigamp_{i=1}^{n}(\phi_i\leq\psi_i)\Rightarrow\varphi\leq\psi$ is true on $\mathbb{B}$ under $\theta$ (notation: $(\mathbb{B},\theta)\vDash\bigamp_{i=1}^{n}(\phi_i\leq\psi_i)\Rightarrow\varphi\leq\psi$) if $\varphi^{(\mathbb{B}, \theta)}\leq\psi^{(\mathbb{B}, \theta)}$ holds whenever $\varphi_i^{(\mathbb{B}, \theta)}\leq\psi_i^{(\mathbb{B}, \theta)}$ holds for every $1\leq i\leq n$, and $\bigamp_{i=1}^{n}(\phi_i\leq\psi_i)\Rightarrow\varphi\leq\psi$ is valid on $\mathbb{B}$ (notation: $\mathbb{B}\vDash\bigamp_{i=1}^{n}(\phi_i\leq\psi_i)\Rightarrow\varphi\leq\psi$) if $(\mathbb{B},\theta)\vDash\bigamp_{i=1}^{n}(\phi_i\leq\psi_i)\Rightarrow\varphi\leq\psi$ for every $\theta$.

Another useful way to look at a formula $\varphi(p_1, \ldots, p_n)$ is to interpret it as an $n$-ary function $\varphi^{\mathbb{B}}:\mathbb{B}^{n}\to\mathbb{B}$ such that $\varphi^{\mathbb{B}}(a_1, \ldots, a_n)=\theta(\varphi)$ where $\theta:\mathsf{Prop}\to\mathbb{B}$ satisfies $\theta(p_i)=a_i$, $i=1, \ldots, n$.

Recall that an element $b\in \mathbb{B}$ is an \emph{atom} if $b\neq\bot$ and for any $c\in B$ s.t. $c\leq b$, either $c=\bot$ or $c=b$.
Moreover, a modal operator $\Box$ on $\mathbb{B}$ is \emph{completely multiplicative} if $\bigwedge\{\Box a\mid a\in A\}$ exists for any $A\subseteq B$ such that $\bigwedge A$ exists, and $\bigwedge\{\Box a\mid a\in A\}=\Box(\bigwedge A)$. A BAO is:
\begin{itemize}
\item[($\mathcal{C}$)] \emph{complete} if the greatest lower bound $\bigwedge A$ and least upper bound $\bigvee A$ exist for any $A\subseteq B$;
\item[($\mathcal{A}$)] \emph{atomic} if for any $a\neq\bot$ there exists some atom $b\in B$ such that $\bot\neq b\leq a$;
\item[($\mathcal{V}$)] \emph{completely multiplicative}\footnote{In \cite{Ho16}, algebras satisfying condition ($\mathcal{V}$) are called completely additive.} if $\Box$ is completely multiplicative.
\end{itemize}
A BAO is a $\mathcal{CV}$-\emph{BAO} if it is complete and completely multiplicative, and is a $\mathcal{CAV}$-\emph{BAO} if it is complete, atomic and completely multiplicative, and other abbreviations are given in a similar way. Notice that the definitions of completeness and atomicity also apply to Boolean algebras.
\end{definition}

As to the correspondence between BAOs and possibility frames, the following definition provides the frame-to-algebra direction:

\begin{definition}
For any possibility frame $\mathbb{F}=(W, \sqsubseteq, R, \mathsf{P})$, let the BAO $\mathsf{P}$ of Definition \ref{adef:poss:frame:model} be the BAO \emph{dual} to $\mathbb{F}$, denoted by $\mathbb{B}_{\mathsf{P}}$. If $\mathbb{F}$ is a full possibility frame, then $\mathsf{P}=\mathsf{RO}(W,\tau_{\sqsubseteq})$ and we refer to $\mathbb{B}_{\mathsf{P}}$ as $\mathbb{B}_{\mathsf{RO}}$ (the \emph{regular open dual BAO} of $\mathbb{F}$).
\end{definition}

\begin{prop}
For any full possibility frame $\mathbb{F}$, $\mathbb{B}_{\mathsf{RO}}$ is a $\mathcal{CV}$-BAO.
\end{prop}

The existence of the frame-to-algebra direction of the duality defined above induces a bijection between valuations on $\mathbb{F}$ and interpretations of propositional variables into $\mathbb{B}_{\mathsf{P}}$.

\begin{definition}
For any possibility frame $\mathbb{F}=(W, \sqsubseteq, R, \mathsf{P})$, any valuation $V$ on $\mathbb{F}$ can be associated with the assignment $\theta_{V}:\mathsf{Prop}\to\mathbb{B}_{\mathsf{P}}$ defined by $\theta_{V}(p):=V(p)$ for every $p\in\mathsf{Prop}$. Conversely, any assignment $\theta:\mathsf{Prop}\to\mathbb{B}_{\mathsf{P}}$ can be associated with the valuation $V_{\theta}:\mathsf{Prop}\to\mathsf{P}$ defined by $V_{\theta}(p)=\theta(p)$ for every $p\in\mathsf{Prop}$.
\end{definition}

It is easy to check the following equivalences hold:

\begin{prop}\label{aprop:equiv:duality}
For any possibility frame $\mathbb{F}=(W, \sqsubseteq, R, \mathsf{P})$, for any valuation $V:\mathsf{Prop}\to\mathsf{P}$ on $\mathbb{F}$, any assignment $\theta:\mathsf{Prop}\to\mathsf{P}$ on $\mathbb{B}_{\mathsf{P}}$, any $w\in W$, any formula $\phi$, any inequality $\phi\leq\psi$, any quasi-inequality $\bigamp_{i=1}^{n}(\phi_i\leq\psi_i)\Rightarrow\varphi\leq\psi$,
\begin{itemize}
\item $\mathbb{F}, V, w\vDash\varphi$ iff $w\in \theta_{V}(\varphi)$;
\item $\mathbb{F}, V_{\theta}, w\vDash\varphi$ iff $w\in \theta(\varphi)$;
\item $\mathbb{F}\vDash\varphi$ iff $\mathbb{B}_{\mathsf{P}}\vDash\varphi$;
\item $\mathbb{F}, V\vDash\phi\leq\psi$ iff $\mathbb{B}_{\mathsf{P}}, \theta_{V}\vDash\phi\leq\psi$;
\item $\mathbb{F}, V_{\theta}\vDash\phi\leq\psi$ iff $\mathbb{B}_{\mathsf{P}}, \theta\vDash\phi\leq\psi$;
\item $\mathbb{F}\vDash\phi\leq\psi$ iff $\mathbb{B}_{\mathsf{P}}\vDash\phi\leq\psi$;
\item $\mathbb{F}, V\vDash\bigamp_{i=1}^{n}(\phi_i\leq\psi_i)\Rightarrow\varphi\leq\psi$ iff $\mathbb{B}_{\mathsf{P}}, \theta_{V}\vDash\bigamp_{i=1}^{n}(\phi_i\leq\psi_i)\Rightarrow\varphi\leq\psi$;
\item $\mathbb{F}, V_{\theta}\vDash\bigamp_{i=1}^{n}(\phi_i\leq\psi_i)\Rightarrow\varphi\leq\psi$ iff $\mathbb{B}_{\mathsf{P}}, \theta\vDash\bigamp_{i=1}^{n}(\phi_i\leq\psi_i)\Rightarrow\varphi\leq\psi$;
\item $\mathbb{F}\vDash\bigamp_{i=1}^{n}(\phi_i\leq\psi_i)\Rightarrow\varphi\leq\psi$ iff $\mathbb{B}_{\mathsf{P}}\vDash\bigamp_{i=1}^{n}(\phi_i\leq\psi_i)\Rightarrow\varphi\leq\psi$.
\end{itemize}
\end{prop}
The duality theoretic facts outlined above make it possible to transfer the development of correspondence theory from frames to algebras, similarly to the way in which algebraic correspondence is developed for Kripke frames. In particular, $\mathbb{B}_{\mathsf{RO}}$ will be the algebra where correspondence theory over full possibility frames is developed. However, the essential difference between the correspondence for Kripke semantics and the present setting is that the algebra $\mathbb{B}_{\mathsf{RO}}$ is not atomic in general (see \cite[Example 2.40]{Ho16}). This implies that some of the rules of the original ALBA-type algorithm (cf.\ \cite{CoGoVa06,CoPa12}) for complex algebras of Kripke frames (namely, the so-called approximation rules which relied on atomicity) are not going to be sound in this setting. The analysis of the semantic environment of the regular open dual BAOs, developed in the next section, will give insights on how to design the algorithm in this semantic setting.

\section{Semantic environment for the language of ALBA}\label{aSec:algebraic:analysis}

In the present section, we will provide the algebraic semantic environment for the correspondence algorithm ALBA in the setting of possibility semantics. We will show the semantic properties which will be used for the interpretation of the expanded modal language of the algorithm ALBA in Section \ref{asec:expanded:modal:language}. The first notable feature of this language is that it includes special variables (besides the propositional variables), the so-called {\em nominals}, which in the original setting are interpreted as the {\em atoms} of the complex algebras of Kripke frames. This interpretation of nominals pivots on the fact that the complex algebras of Kripke frames are atomic, that is, are completely join-generated by their atoms. Likewise, in order to define a suitable interpretation for the nominals in the possibility setting, we need to find a class of elements which join-generate the complex algebra of any full possibility frame. Towards this goal, our strategy will consist in defining a BAO $\mathbb{B}_{\mathsf{K}}$ in which the BAO $\mathbb{B}_{\mathsf{RO}}$ can be order-embedded. This algebra will be used as an auxiliary tool to show that every element in $\mathbb{B}_{\mathsf{RO}}$ can be represented as the join of regular open closures of atoms in $\mathbb{B}_{\mathsf{K}}$. Hence, this will show that the regular open closures of atoms in $\mathbb{B}_{\mathsf{K}}$ are a suitable class of interpretants for nominal variables. 

The second notable feature of the expanded language of ALBA is that it includes additional modal operators interpreted as the adjoints of the modal operators of the original language. In what follows, we will show that also these connectives have a natural interpretation in $\mathbb{B}_{\mathsf{RO}}$.

\subsection{The auxiliary BAO $\mathbb{B}_{\mathsf{K}}$}

Clearly, any full possibility frame $\mathbb{F}_{\mathsf{RO}}=(W, \sqsubseteq, R, \mathsf{RO}(W,\tau_{\sqsubseteq}))$ can be associated with the {\em bimodal frame} $\mathbb{F}_{\mathsf{K}}=(W, \sqsubseteq, R)$, the complex algebra of which is the bimodal BAO $\mathbb{B}_{\mathsf{K}}$ (cf.\ \cite{vBBeHo16}).

Diagrammatically, the $\mathcal{CAV}$-BAO $\mathbb{B}_{\mathsf{K}}$ dually corresponds to the Kripke frame $\mathbb{F}_{\mathsf{K}}$, and the $\mathcal{CV}$-BAO $\mathbb{B}_{\mathsf{RO}}$ to the full possibility frame $\mathbb{F}_{\mathsf{RO}}$, $e$ is the order-embedding which sends a regular open subset in $\mathbb{B}_{\mathsf{RO}}$ to itself in $\mathbb{B}_{\mathsf{K}}$, and $U$ sends a full possibility frame to its underlying bimodal Kripke frame ``forgetting'' the algebra $\mathsf{RO}(W,\tau_{\sqsubseteq})$ (and hence the restriction on the admissible valuations)

\begin{center}
\begin{tikzpicture}[node/.style={circle, draw, fill=black}, scale=1]\label{atable:U:shape}
\node (BRO) at (-1.5,-1.5) {$\mathbb{B}_{\mathsf{RO}}$};
\node (Bfull) at (-1.5,1.5) {$\mathbb{B}_{\mathsf{K}}$};
\node (FP) at (1.5,-1.5) {$\mathbb{F}_{\mathsf{RO}}$};
\node (FK) at (1.5,1.5) {$\mathbb{F}_{\mathsf{K}}$};
\draw [right hook->] (BRO) to node[left]{$e$} (Bfull);
\draw [->] (FP) to node[right]{$U$} (FK);
\draw [<->] (BRO) to node[above] {$\cong^{\partial}$} (FP);
\draw [<->] (Bfull) to node[above] {$\cong^{\partial}$} (FK);
\end{tikzpicture}
\end{center}

The formal definition of the BAO $\mathbb{B}_{\mathsf{K}}$ is reported below:

\begin{definition}[Full dual Boolean algebra with operators]
For any full possibility frame $\mathbb{F}=(W, \sqsubseteq, R, \mathsf{RO}(W,\tau_{\sqsubseteq}))$, the \emph{full dual Boolean algebra with operators (full dual BAO)} $\mathbb{B}_{\mathsf{K}}$\footnote{Notice that here the ``full'' means that the carrier set is the powerset of $W$, rather than the ``full'' in ``full possibility frame''.} is defined as $\mathbb{B}_{\mathsf{K}}=(P(W), \emptyset, W, \cap, \cup, -, \Box_{\mathsf{K}}, \Box_{\sqsubseteq})$, where $\cap, \cup, -$ are set-theoretic intersection, union and complementation respectively, $\Box_{\mathsf{K}}(a)=\{w\in W\mid R[w]\subseteq a\}$, and $\Box_{\sqsubseteq}(a)=\{w\in W\mid (\forall v\sqsubseteq w)(v\in a)\}$.
\end{definition}

It is easy to see that $\mathbb{B}_{\mathsf{K}}$ is a complete \emph{atomic} Boolean algebra with complete operators. It is also clear that the carrier set of the regular open dual BAO is a subset of the full dual BAO, hence the natural embedding $e:\mathbb{B}_{\mathsf{RO}}\hookrightarrow\mathbb{B}_{\mathsf{K}}$ is well defined and is an order-embedding. Notice that by Proposition \ref{afacts:regular:open}, arbitrary intersections of regular open sets in a downset topology are again regular open, therefore $e$ is completely meet-preserving. Notice also that $\Box_{\mathsf{RO}}$ is the restriction of $\Box_{\mathsf{K}}$ to $\mathbb{B}_{\mathsf{RO}}$. All these observations can be summarized as follows:

\begin{lemma}\label{alemma:preserve:e}
$e:\mathbb{B}_{\mathsf{RO}}\hookrightarrow\mathbb{B}_{\mathsf{K}}$ is a completely meet-preserving order-embedding such that $e\circ \Box_{\mathsf{RO}}=\Box_{\mathsf{K}}\circ e$.
\end{lemma}

However, it is important to stress that, since $\mathbb{B}_{\mathsf{RO}}$ and $\mathbb{B}_{\mathsf{K}}$ have different definitions of join and complementation, $\mathbb{B}_{\mathsf{RO}}$ is \emph{not} a subalgebra of $\mathbb{B}_{\mathsf{K}}$. 

The next corollary follows immediately from the previous lemma (see e.g.\ \cite[Proposition 7.34]{DaPr90}):

\begin{cor}\label{acor:existence:left:adjoint}
$e:\mathbb{B}_{\mathsf{RO}}\hookrightarrow\mathbb{B}_{\mathsf{K}}$ has a left adjoint $c:\mathbb{B}_{\mathsf{K}}\to\mathbb{B}_{\mathsf{RO}}$ defined, for every $a\in \mathbb{B}_{\mathsf{K}}$,
\begin{center}
$c(a)=\bigwedge_{\mathsf{RO}}\{b\in\mathbb{B}_{\mathsf{RO}}\mid a\leq e(b)\}.$
\end{center}
\end{cor}
Clearly, $c(X)=\mathsf{ro}(X)$ for any $X\subseteq W$, Indeed, by definition $c(X)=\bigwedge_{\mathsf{RO}}\{Y\in\mathsf{RO}(W,\tau_{\sqsubseteq})\mid X\leq e(Y)\}=\bigcap\{Y\in\mathsf{RO}(W,\tau_{\sqsubseteq})\mid X\subseteq Y\}$, which is the least regular open set containing $X$. The closure operator $c$ will be referred to as the \emph{regular open closure map} and $c(a)$ as the \emph{regular open closure} of $a$. We let $\mathsf{PsAt}(\mathbb{B}_{\mathsf{RO}}):=\{c(x)\mid x\in\mathsf{At}(\mathbb{B}_{\mathsf{K}})\}$ (here $\mathsf{PsAt}$ stands for pseudo-atom, and $\mathsf{At}(\mathbb{B})$ denotes the set of atoms in the BAO $\mathbb{B}$) be the set of regular open closures of atoms in $\mathbb{B}_{\mathsf{K}}$, which will be shown to be the join-generators of $\mathbb{B}_{\mathsf{RO}}$.

\subsubsection{A class of interpretants for nominals}
As mentioned ealy on, the key requirement for a suitable class of interpretants for nominals is that it is join-dense in $\mathbb{B}_{\mathsf{RO}}$. In what follows, we give a proof of this property for $\mathsf{PsAt}(\mathbb{B}_{\mathsf{RO}})$. This result has already been proved in \cite[Lemma 3.7]{Ya16}; we give an alternative proof in the dual algebraic setting. Let us preliminarily recall that, by general facts of the theory of closure operators on posets, $c\circ e = Id_{\mathbb{B}_{\mathsf{RO}}}$.

\begin{prop}\label{aprop:regular:open:join}
For any $a\in \mathbb{B}_{\mathsf{RO}}$, $$a=\bigvee_{\mathsf{RO}}\{c(x)\mid x\in \mathsf{At}(\mathbb{B}_{\mathsf{K}})\mbox{ and }x\leq e(a)\}=\bigvee_{\mathsf{RO}}\{y\in\mathsf{PsAt}(\mathbb{B}_{\mathsf{RO}})\mid y\leq a\}.$$
\end{prop}
\begin{proof}
The first equality follows from the fact that $c\circ e = Id_{\mathbb{B}_{\mathsf{RO}}}$ and that left adjoint preserve arbitrary existing joins; the second equality, from the definition of $\mathsf{PsAt}(\mathbb{B}_{\mathsf{RO}})$ and adjunction between $c$ and $e$.
\end{proof}

\subsection{Interpreting the additional connectives of the expanded language of ALBA}

As mentioned early on, $\mathbb{B}_{\mathsf{RO}}$ and $\mathbb{B}_{\mathsf{K}}$ are both complete, and $\Box_{\mathsf{RO}}:\mathbb{B}_{\mathsf{RO}}\to\mathbb{B}_{\mathsf{RO}}$ and $\Box_{\mathsf{K}}:\mathbb{B}_{\mathsf{K}}\to\mathbb{B}_{\mathsf{K}}$ are both completely meet-preserving. Thus both of them have left adjoints, which are denoted by $\Diamondblack_{\mathsf{RO}}$ and $\Diamondblack_{\mathsf{K}}$, respectively. They will be used as the semantic interpretation of the additional connective in the expanded modal language in the next section. In what follows we will explicitly compute the definitions of the adjoints.

\begin{lemma}\label{alem:Diamond:Full}
$\Diamondblack_{\mathsf{K}}(X)=R[X]$ for any $X\subseteq W$.
\end{lemma}
\begin{proof}
By definition of adjunction, for any $Y\in P(W)$, $w\in W$,
\begin{center}
\begin{tabular}{r c l}
$\Diamondblack_{\mathsf{K}}(\{w\})\subseteq Y$ & iff & $\{w\}\subseteq\Box_{\mathsf{K}}(Y)$\\
& iff & $\{w\}\subseteq \{v\in W\mid R[v]\subseteq Y\}$\\
& iff & $R[w]\subseteq Y$,\\
\end{tabular}
\end{center}
Therefore $\Diamondblack_{\mathsf{K}}(\{w\})=R[w]$. Since left adjoints preserve existing joins,

\begin{center}
\begin{tabular}{r c l}
$\Diamondblack_{\mathsf{K}}(X)$ & = & $\Diamondblack_{\mathsf{K}}(\bigcup\{\{w\}\mid w\in X\})$\\
& = & $\bigcup\{\Diamondblack_{\mathsf{K}}(\{w\})\mid w\in X\}$\\
& = & $\bigcup\{R[w]\mid w\in X\}$\\
& = & $R[X]$.
\end{tabular}
\end{center}
\end{proof}

\begin{lemma}\label{alem:Diamond:RO}
$\Diamondblack_{\mathsf{RO}}(a)=(c\circ\Diamondblack_{\mathsf{K}}\circ e)(a)$.
\end{lemma}
\begin{proof}
We have the following chain of equalities:
\begin{center}
\begin{tabular}{r c l l}
$\Diamondblack_{\mathsf{RO}}(a)$ & = & $\bigwedge_{\mathsf{RO}}\{b\in\mathbb{B}_{\mathsf{RO}}\mid a\leq\Box_{\mathsf{RO}}(b)\}$ & (adjunction property)\\
& = &$\bigwedge_{\mathsf{RO}}\{b\in\mathbb{B}_{\mathsf{RO}}\mid e(a)\leq(\Box_{\mathsf{K}}\circ e)(b)\}$ & (Lemma \ref{alemma:preserve:e})\\
& = &$\bigwedge_{\mathsf{RO}}\{b\in\mathbb{B}_{\mathsf{RO}}\mid (\Diamondblack_{\mathsf{K}}\circ e)(a)\leq e(b)\}$ & (definition of adjunction)\\
& = &$\bigwedge_{\mathsf{RO}}\{b\in\mathbb{B}_{\mathsf{RO}}\mid (c\circ\Diamondblack_{\mathsf{K}}\circ e)(a)\leq b\}$ & (definition of adjunction)\\
& = &$(c\circ\Diamondblack_{\mathsf{K}}\circ e)(a)$.
\end{tabular}
\end{center}
\end{proof}
\begin{cor}
$\Diamondblack_{\mathsf{RO}}(X)=\mathsf{ro}(R[X])$ for any $X\in\mathsf{RO}(W,\tau_{\sqsubseteq})$.
\end{cor}
\begin{proof}
We have the following chain of equalities:
\begin{center}
\begin{tabular}{r c l l}
$\Diamondblack_{\mathsf{RO}}(X)$ & = &$(c\circ\Diamondblack_{\mathsf{K}}\circ e)(X)$ & (Lemma \ref{alem:Diamond:RO})\\
& = &$(c\circ\Diamondblack_{\mathsf{K}})(X)$ &\\
& = &$c(R[X])$ & (Lemma \ref{alem:Diamond:Full})\\
& = &$\mathsf{ro}(R[X])$. & (Corollary \ref{acor:existence:left:adjoint})
\end{tabular}
\end{center}
\end{proof}

\section{Preliminaries on algorithmic correspondence}\label{aSec:expanded:language}

Our algorithmic correspondence and canonicity proofs are based on duality and order-theoretic insights, and distill the order-theoretic properties from the concrete semantic setting. We will explain how we work for the correspondence over full possibility frames. This methodology works in a similar way in the setting of filter-descriptive possibility frames.

As we have already seen from Section \ref{aSec:Algebraic:Semantics}, we have the dual equivalence between the full possibility frame $\mathbb{F}_{\mathsf{RO}}$ and the regular open dual BAO $\mathbb{B}_{\mathsf{RO}}$. Now we can understand the sketch of algorithmic correspondence as follows (here we abuse notation to mix syntactic symbols and semantic objects, for the convenience of understanding):
\begin{center}
\begin{tabular}{l c l}\label{atable:U:shape}
$\mathbb{F}_{\mathsf{RO}}\vDash\varphi(\vec p)$ & &$\mathbb{F}_{\mathsf{RO}}\vDash$FO(Pure$(\varphi(\vec p))$)\\
\\
\ \ \ \ \ \ $\Updownarrow$ & &\ \ \ \ \ \ $\Updownarrow$\\
\\
$\mathbb{B}_{\mathsf{RO}}\vDash(\forall \vec{p}\in \mathbb{B}_{\mathsf{RO}})(\phi(\vec p))$ &\ \ \ $\Leftrightarrow$ \ \ \ &$\mathbb{B}_{\mathsf{RO}}\vDash(\forall\vec{\mathbf{i}}\in \mathsf{PsAt}(\mathbb{B}_{\mathsf{RO}}))$Pure$(\phi(\vec p))$
\end{tabular}
\end{center}

For the left arm of the ``U-shaped'' argument given above, the validity of $\phi(\vec p)$ is understood as $(\forall \vec p\in \mathbb{B}_{\mathsf{RO}})(\phi(\vec p))$, where the interpretations of propositional variables range over elements in $\mathbb{B}_{\mathsf{RO}}$. The algorithm transforms $\phi(\vec p)$ into an equivalent set of pure quasi-inequalities Pure$(\phi(\vec p))$ which does not contain propositional variables, but only nominals\footnote{In lattice-based logic settings, there is another kind of variables called co-nominals (see e.g.\ \cite{CoPa12}), which are interpreted as co-atoms or complete meet-irreducibles. Since in the Boolean setting, co-nominals can be interpreted as the negation of nominals, they are not really necessary here. In the remainder of the paper, we will not use co-nominals.}, and the validity of Pure$(\phi(\vec p))$ can be understood as $(\forall \vec{\mathbf{i}}\in \mathsf{PsAt}(\mathbb{B}_{\mathsf{RO}}))$(Pure$(\phi(\vec p))$), where $\mathsf{PsAt}(\mathbb{B}_{\mathsf{RO}})$ is the collection of regular open closures of atoms in $\mathbb{B}_{\mathsf{K}}$ rather than atoms in $\mathbb{B}_{\mathsf{RO}}$, which serves as the set of join-generators in $\mathbb{B}_{\mathsf{RO}}$. This interpretation is the key to the soundness of the approximation rules, in the spirit of in \cite[Lemma 3.7]{Ya16}. Once the nominals and the other symbols in Pure$(\phi(\vec p))$ have first-order translations, we have the first-order translation FO(Pure$(\phi(\vec p))$) of the pure quasi-inequalities, which is the first-order correspondent of $\phi(\vec p)$ over full possibility frames.

Therefore, the ingredients for the algorithmic correspondence proof to go through can be listed as follows:
\begin{itemize}
\item An expanded modal language as the syntax of the algorithm, as well as their interpretations in $\mathbb{B}_{\mathsf{RO}}$;
\item An algorithm which transforms a given modal formula $\phi(\vec p)$ into equivalent pure quasi-inequalities Pure$(\phi(\vec p))$;
\item A soundness proof of the algorithm with respect to $\mathbb{B}_{\mathsf{RO}}$;
\item A syntactically identified class of formulas on which the algorithm is successful;
\item A first-order correspondence language and first-order translation which transforms pure quasi-inequalities into their equivalent first-order correspondents.
\end{itemize}

In the remainder of the paper, we will define an expanded modal language which the algorithm will manipulate (Section \ref{asec:expanded:modal:language}), define the first-order correspondence language of possibility frames (Section \ref{asec:correspondence:language}) and the counterpart of the standard translation into this language, which we refer to as {\em regular open translation} (Section \ref{asec:regular:open:translation}). We report on the definition of Sahlqvist and inductive formulas (Section \ref{asec:Sahlqvist}), and define a modified version of the algorithm ALBA suitable for the possibility semantic environment (Section \ref{aSec:ALBA}), show its soundness over regular open dual BAOs (Section \ref{aSec:soundness}) and its success (Section \ref{aSec:success}) on Sahlqvist and inductive formulas. In Section \ref{aSec:Canonicity} we show the soundness of the algorithm over the dual BAOs of filter-descriptive possibility frames.

\subsection{The expanded modal language and the regular open translation}

In the present section, we will define the expanded modal language for the algorithm, the first-order and second-order correspondence language, as well as the regular open translation of the expanded modal language into the correspondence language. We will also show that the translation preserves truth conditions.

\subsubsection{The expanded modal language $\mathcal{L}^{+}$}\label{asec:expanded:modal:language}

The expanded modal language $\mathcal{L}^{+}$ is a proper expansion of the modal language. Apart from the propositional variables and connectives in the modal language, there are also a set $\mathsf{Nom}$ of \emph{nominals}, a special kind of variables to be interpreted as elements in $\mathsf{PsAt}(\mathbb{B}_{\mathsf{RO}})$, and the \emph{black connectives} $\Diamondblack,\blacksquare$, i.e.\ the unary connectives to be interpreted as the adjoints of $\Box$ and $\Diamond$ respectively. The formal definition of the formulas in the expanded modal language $\mathcal{L}^{+}$ is given as follows:
$$\varphi::=p \mid \nomi \mid \neg\varphi \mid \varphi\land\varphi \mid \Box\varphi \mid \Diamondblack\varphi,$$
where $p\in\mathsf{Prop}$ and $\nomi\in\mathsf{Nom}$. We also define $\blacksquare\varphi:=\neg\Diamondblack\neg\varphi$, and the other abbreviations are defined similar to the basic modal language. In the algorithm, it will be convenient to use the abbreviations as primitive symbols in the definition of the rules.

In order to interpret the expanded modal language on possibility frames and the dual BAOs, we need to extend the valuation $V$ and assignment $\theta$ also to nominals. As we have already seen in Section \ref{aSec:algebraic:analysis}, every element in $\mathbb{B}_{\mathsf{RO}}$ can be represented as the join of elements in $\mathsf{PsAt}(\mathbb{B}_{\mathsf{RO}})\subseteq\mathbb{B}_{\mathsf{RO}}=\mathsf{RO}(W,\tau_{\sqsubseteq})$. Therefore, we are going to interpret the nominals as elements in $\mathsf{PsAt}(\mathbb{B}_{\mathsf{RO}})$, i.e.\ $V(\nomi), \theta(\nomi)\in\mathsf{PsAt}(\mathbb{B}_{\mathsf{RO}})\subseteq\mathsf{RO}(W,\tau_{\sqsubseteq})$.

The satisfaction relation for the additional symbols will be given as follows in any possibility frame $\mathbb{F}=(W,\sqsubseteq,R,\mathsf{P})$:
\begin{definition}
\begin{itemize}
\item $\mathbb{F}, V, w\vDash \nomi$ iff $w\in V(\nomi)$;
\item $\mathbb{F}, V, w\vDash \Diamondblack\varphi$ iff $(\forall v\sqsubseteq w)(\exists u\sqsubseteq v)(\exists t\sqsupseteq u)\exists s(Rst$ and $\mathbb{F}, V, s\vDash \varphi)$.
\end{itemize}
\end{definition}

Notice that $V(\nomi)$ is not necessarily in $\mathsf{P}$, but it is always in $\mathsf{RO}(W,\tau_{\sqsubseteq})$. Similarly, $\mathsf{P}$ is not necessarily closed under $\Diamondblack_{\mathsf{RO}}$, but $\mathsf{RO}(W,\tau_{\sqsubseteq})$ is. Therefore, when interpreting formulas in the expanded modal language, we only restrict the interpretations of propositional variables to $\mathsf{P}$ (therefore also all formulas in the basic modal language), and allow formulas in the expanded modal language to be interpreted in $\mathsf{RO}(W, \tau_{\sqsubseteq})$. Truth set and validity are defined similarly to the basic modal language.

It is easy to check that the following facts hold for the expanded modal language:

\begin{prop}\label{aprop:semantic:condition:black}
\begin{itemize}
\item $\mathbb{F}, V, w\vDash \Diamondblack\varphi$ iff $w\in \mathsf{ro}(R[\llbracket\phi\rrbracket^{\mathbb{F}, V}])$;
\item $\mathbb{F}, V, w\vDash\blacksquare\varphi$ iff $(\forall v\sqsubseteq w)(\forall u\sqsupseteq v)\forall t(Rtu\Rightarrow(\exists s\sqsubseteq t)(\mathbb{F}, V, s\vDash \varphi))$;
\item $\llbracket\Diamondblack\phi\rrbracket^{\mathbb{F}, V}=\Diamondblack_{\mathsf{RO}}\llbracket\phi\rrbracket^{\mathbb{F}, V}$;
\item $\llbracket\blacksquare\phi\rrbracket^{\mathbb{F}, V}=(-_{\mathsf{RO}}\circ\Diamondblack_{\mathsf{RO}}\circ -_{\mathsf{RO}})(\llbracket\phi\rrbracket^{\mathbb{F}, V})$.
\end{itemize}

\end{prop}
The next proposition shows that $\blacksquare$ is interpreted as the right adjoint of $\Diamond$:
\begin{prop}
For any $X, Y\in\mathsf{RO}(W,\tau_{\sqsubseteq})$, $$(-_{\mathsf{RO}}\circ\Box_{\mathsf{RO}}\circ -_{\mathsf{RO}})(X)\subseteq Y\mbox{ iff }X\subseteq (-_{\mathsf{RO}}\circ\Diamondblack_{\mathsf{RO}}\circ -_{\mathsf{RO}})(Y).$$
\end{prop}

For the algebraic semantics of the expanded modal language, we use a kind of hybrid algebraic structures obtained from arbitrary possibility frames: consider a possibility frame $\mathbb{F}=(W,\sqsubseteq,R,\mathsf{P})$ (whose underlying full possibility frame is well-defined) and its underlying full possibility frame $\mathbb{F}=(W,\sqsubseteq,R,\mathsf{RO}(W,\tau_{\sqsubseteq}))$, we dualize the latter to obtain the regular open dual BAO $\mathbb{B}_{\mathsf{RO}}=(\mathsf{RO}(W,\tau_{\sqsubseteq}), \emptyset, W, \wedge_{\mathsf{RO}}, \vee_{\mathsf{RO}}, -_{\mathsf{RO}}, \Box_{\mathsf{RO}})$, and put an admissible set $\mathsf{P}$ on top of it and get a \emph{hybrid dual BAO} $(\mathbb{B}_{\mathsf{RO}}, \mathsf{P})$, which restricts the assignment of propositional variables to $\mathsf{P}$, but still allows formulas in the expanded modal language to range over $\mathsf{RO}(W,\tau_{\sqsubseteq})$. For the interpretation of the expanded modal language, we require that $\theta(\nomi)\in\mathsf{PsAt}(\mathbb{B}_{\mathsf{RO}})$ and $\Diamondblack$ is interpreted as $\Diamondblack_{\mathsf{RO}}:\mathbb{B}_{\mathsf{RO}}\to\mathbb{B}_{\mathsf{RO}}$. For the definition of validity, we use the notation $\mathbb{B}_{\mathsf{RO}}\vDash_{\mathsf{P}}\varphi$ to indicate that the assignments of propositional variables range over $\mathsf{P}$ (while the assignments of nominals range over $\mathsf{PsAt}(\mathbb{B}_{\mathsf{RO}})$). 

It is easy to check that Proposition \ref{aprop:equiv:duality} generalizes to the expanded modal language in the case of full possibility frames. For the case of arbitrary possibility frames, we use the hybrid dual BAO $(\mathbb{B}_{\mathsf{RO}}, \mathsf{P})$ and the adapted version of validity $\mathbb{B}_{\mathsf{RO}}\vDash_{\mathsf{P}}\varphi$. The discussion above can be summarized as follows:

\begin{prop}\label{aprop:equiv:duality:expanded}
For any formula $\phi$ in the expanded language, 

\begin{itemize}
\item For any full possibility frame $\mathbb{F}_{\mathsf{RO}}$ and its dual BAO $\mathbb{B}_{\mathsf{RO}}$,
\begin{itemize}
\item $\mathbb{F}_{\mathsf{RO}}\vDash\varphi$ iff $\mathbb{B}_{\mathsf{RO}}\vDash\varphi$;
\item $\mathbb{F}_{\mathsf{RO}}\vDash\phi\leq\psi$ iff $\mathbb{B}_{\mathsf{RO}}\vDash\phi\leq\psi$;
\item $\mathbb{F}_{\mathsf{RO}}\vDash\bigamp_{i=1}^{n}(\phi_i\leq\psi_i)\Rightarrow\varphi\leq\psi$ iff $\mathbb{B}_{\mathsf{RO}}\vDash\bigamp_{i=1}^{n}(\phi_i\leq\psi_i)\Rightarrow\varphi\leq\psi$.
\end{itemize}
\item For any possibility frame $\mathbb{F}_{\mathsf{P}}$ and its hybrid dual BAO $(\mathbb{B}_{\mathsf{RO}}, \mathsf{P})$, 
\begin{itemize}
\item $\mathbb{F}_{\mathsf{P}}\vDash\varphi$ iff $\mathbb{B}_{\mathsf{RO}}\vDash_{\mathsf{P}}\varphi$;
\item $\mathbb{F}_{\mathsf{P}}\vDash\phi\leq\psi$ iff $\mathbb{B}_{\mathsf{RO}}\vDash_{\mathsf{P}}\phi\leq\psi$;
\item $\mathbb{F}_{\mathsf{P}}\vDash\bigamp_{i=1}^{n}(\phi_i\leq\psi_i)\Rightarrow\varphi\leq\psi$ iff $\mathbb{B}_{\mathsf{RO}}\vDash_{\mathsf{P}}\bigamp_{i=1}^{n}(\phi_i\leq\psi_i)\Rightarrow\varphi\leq\psi$.
\end{itemize}
\end{itemize}
\end{prop}

\subsubsection{The correspondence languages}\label{asec:correspondence:language}

In order to express the first-order correspondents of modal formulas, we need to define the first-order and second-order correspondence language $\mathcal{L}^1$ and $\mathcal{L}^2$. The first-order correspondence language $\mathcal{L}^{1}$ consists of a set of unary predicate symbols $P_n$, each of which corresponds to a propositional variable $p_n$, two binary relation symbols $\sqsubseteq$ and $R$ corresponding to the refinement relation and the accessibility relation respectively, a set of individual symbols $i_n$, each of which corresponds to a nominal $\nomi_n$, and the quantifiers $\forall x, \exists x$ are first-order, i.e.\ ranging over individual variables. The second-order correspondence language $\mathcal{L}^{2}$ contains all the symbols from $\mathcal{L}^{1}$ as well as second-order quantifiers $\forall^{\mathsf{P}}P, \exists^{\mathsf{P}}P$ over unary predicate variables. In addition, unary predicate symbols are interpreted as admissible subsets, and the second-order quantifiers range over admissible subsets.

The semantic structures to interpret the first-order and second-order formulas are the possibility models $(\mathbb{F}, V)=(W, \sqsubseteq, R, \mathsf{P}, V)$, where an individual symbol $i_n$ is interpreted as a state $\underline{i_n}\in W$ such that $\mathsf{ro}(\{\underline{i_n}\})=V(\nomi_n)$, a unary predicate symbols $P_n$ is interpreted as $V(p_n)\in\mathsf{P}$, and the binary relation symbols $\sqsubseteq$ and $R$ are interpreted as the refinement relation and the accessibility relation denoted by the same symbol, respectively. At the level of possibility frames, we will abuse notation to take the unary predicate symbols $P_n$ and the individual symbols $i_n$ as variables and use quantifiers over them. We use $\llbracket\alpha(\vec x)\rrbracket^{\mathbb{F}, V}$ to denote the $n$-tuples $\vec w\in W^{n}$ that make $\alpha(\vec x)$ true in the model $(\mathbb{F}, V)$, i.e.\ $\llbracket\alpha(\vec x)\rrbracket^{\mathbb{F}, V}:=\{{\vec w\in W^{n}}\mid\mathbb{F}, V\vDash\alpha[\vec w]\}$, which is called the \emph{truth set} of $\alpha(\vec x)$ in $(\mathbb{F}, V)$.

In the definition of correspondence between a modal formula and a first-order formula, we will require the first-order formula to contain only binary relation symbols, and not contain any unary predicate symbol. Now the definition can be given as follows in the setting of full possibility frames:

\begin{definition}
We say that a modal formula $\phi$ in the basic modal language $\mathcal{L}$ \emph{locally corresponds} to a first-order formula $\alpha(x)$ in the first-order correspondence language with no occurence of unary predicate symbols, if for any full possibility frame $\mathbb{F}=(W, \sqsubseteq, R, \mathsf{RO}(W, \tau_{\sqsubseteq}))$, any $w\in W$, $\mathbb{F}, w\vDash\phi$ iff $\mathbb{F}\vDash\alpha[w]$. We say that $\phi$ \emph{globally corresponds} to a first-order sentence $\alpha$ with no occurence of unary predicate symbols, if for any full possibility frame $\mathbb{F}$, $\mathbb{F}\vDash\phi$ iff $\mathbb{F}\vDash\alpha$. For inequalities and quasi-inequalities, the definition is similar.
\end{definition}

The definition above can be easily adapted to the setting of filter-descriptive possibility frames and so on.

\subsubsection{The regular open translation}\label{asec:regular:open:translation}

In the present section, we will give the first-order translation of the expanded modal language into the first-order correspondence language, in the spirit of the standard translation in \cite[Section 2.4]{BRV01}. Since the translation is based on the semantic interpretation of modal formulas on the relational structures, and in possibility semantics, the conditions about regular open sets play an important role, we will call our translation \emph{regular open translation}.

For the sake of convenience, we give the following definition:

\begin{definition}[Syntactic regular open closure](cf.\ \cite[page 8]{Ya16})
Given a first-order formula $\alpha(x)$ with at most $x$ free, the \emph{syntactic regular open closure} $\mathsf{RO}_{x}(\alpha(x))$ is defined as $(\forall y\sqsubseteq x)(\exists z\sqsubseteq y)(\exists z'\sqsupseteq z)\alpha(z')$.
\end{definition}
It is easy to see that the syntactic regular open closure of a formula is interpreted as the semantic regular open closure of its corresponding truth set:
\begin{prop}(cf.\ \cite[Lemma 3.8]{Ya16})
$\llbracket\mathsf{RO}_x(\alpha(x))\rrbracket^{\mathbb{F}, V}=\mathsf{ro}(\llbracket\alpha(x)\rrbracket^{\mathbb{F}, V})$.
\end{prop}

Since nominals are interpreted as elements in $\mathsf{PsAt}(\mathbb{B}_{\mathsf{RO}})$, i.e.\ regular open closures of singletons, we will translate nominals to the syntactic regular open closure of the identity $i=x$, where $i$ is the individual symbol interpreted as a state $\underline{i_n}\in W$ such that $\mathsf{ro}(\{\underline{i_n}\})=V(\nomi_n)$. The connectives are interpreted according to the definition of their satisfaction relations.

Now we are ready to give the regular open translation as follows:

\begin{definition}[Regular open translation](cf.\ \cite[Definition 2.6]{Ya16})
The regular open translation of a formula in the expanded modal language $\mathcal{L}^{+}$ into the first-order correspondence language $\mathcal{L}^{1}$ is given as follows:

\begin{center}
\begin{tabular}{r c l}
$ST_x(\nomi)$ & := & $\mathsf{RO}_{x}(i=x)$;\\
$ST_x(p_{i})$ & := & $P_{i}x$;\\
$ST_x(\neg\varphi)$ & := & $\forall y(y\sqsubseteq x\to\neg ST_{y}(\varphi))$;\\
$ST_x(\varphi_1\land\varphi_2)$ & := & $ST_x(\varphi_1)\land ST_x(\varphi_2)$;\\
$ST_x(\Box\varphi)$ & := & $\forall y(Rxy\to ST_y(\varphi))$;\\
$ST_x(\Diamondblack\varphi)$ & := & $\mathsf{RO}_{x}(\exists y(Ryx\land ST_y(\varphi)))$.\\
\end{tabular}
\end{center}
\end{definition}

By Proposition \ref{aprop:additional:connectives} and \ref{aprop:semantic:condition:black}, we can also take $\top,\bot,\lor,\to,\Diamond,\blacksquare$ as primitive connectives and define the following translation:

\begin{definition}[Regular open translation continued]\label{adef:trans:2}
$\ $
\begin{center}
\begin{tabular}{r c l}
$ST_x(\top)$ & := & $\top$;\\
$ST_x(\bot)$ & := & $\bot$;\\
$ST_x(\varphi_1\lor\varphi_2)$ & := & $(\forall y\sqsubseteq x)(\exists z\sqsubseteq y)(ST_z(\varphi_1)\lor ST_z(\varphi_2))$;\\
$ST_x(\varphi_1\to\varphi_2)$ & := & $(\forall y\sqsubseteq x)(ST_y(\varphi_1)\to ST_y(\varphi_2))$;\\
$ST_x(\Diamond\varphi)$ & := & $(\forall y\sqsubseteq x)\exists z(Ryz\land(\exists w\sqsubseteq z)(ST_w(\varphi))$;\\
$ST_x(\blacksquare\varphi)$ & := & $(\forall y\sqsubseteq x)(\forall z\sqsupseteq y)\forall w(Rwz\Rightarrow(\exists v\sqsubseteq w)(ST_v(\varphi))$.\\
\end{tabular}
\end{center}
\end{definition}

The following proposition justifies the translation defined above:

\begin{prop}\label{aprop:translation}(cf.\ \cite[Lemma 2.8]{Ya16})
For any possibility frame $\mathbb{F}=(W,\sqsubseteq,R,\mathsf{P})$, any valuation $V$ on $\mathbb{F}$, any $w\in W$ and any formula $\phi(\vec p)$ in $\mathcal{L}^{+}$,
\begin{itemize}
\item $\mathbb{F},V,w\vDash\phi(\vec p)\mbox{ iff }\mathbb{F},V\vDash ST_x(\phi(\vec p))[w];$
\item $\mathbb{F},V\vDash\phi(\vec p)\mbox{ iff }\mathbb{F},V\vDash \forall xST_x(\phi(\vec p));$
\item $\mathbb{F},w\vDash\phi(\vec p)\mbox{ iff }\mathbb{F}\vDash \forall^{\mathsf{P}}{\vec P}\forall{\vec i} ST_x(\phi(\vec p))[w];$
\item $\mathbb{F}\vDash\phi(\vec p)\mbox{ iff }\mathbb{F}\vDash\forall^{\mathsf{P}}{\vec P}\forall{\vec i}\forall xST_x(\phi(\vec p));$
\item $\mathbb{F},V\vDash\phi(\vec p)\leq\psi(\vec p)\mbox{ iff }\mathbb{F},V\vDash \forall x(ST_x(\phi(\vec p))\to ST_x(\psi(\vec p)));$
\item $\mathbb{F}\vDash\phi(\vec p)\leq\psi(\vec p)\mbox{ iff }\mathbb{F}\vDash\forall^{\mathsf{P}}{\vec P}\forall{\vec i}\forall x(ST_x(\phi(\vec p))\to ST_x(\psi(\vec p)));$
\item $\mathbb{F},V\vDash\bigamp_{j=1}^{n}(\phi_j(\vec p)\leq\psi_j(\vec p))\Rightarrow\varphi(\vec p)\leq\psi(\vec p)\mbox{ iff }\mathbb{F},V\vDash 
\bigwedge_{j=1}^{n}\forall x(ST_x(\phi_j(\vec p))\to ST_x(\psi_j(\vec p)))\to\forall x(ST_x(\phi(\vec p))\to ST_x(\psi(\vec p)));$
\item $\mathbb{F}\vDash\bigamp_{j=1}^{n}(\phi_j(\vec p)\leq\psi_j(\vec p))\Rightarrow\varphi(\vec p)\leq\psi(\vec p)\mbox{ iff }\mathbb{F},V\vDash 
\forall^{\mathsf{P}}{\vec P}\forall{\vec i}(\bigwedge_{j=1}^{n}\forall x(ST_x(\phi_j(\vec p))\to ST_x(\psi_j(\vec p)))\to\forall x(ST_x(\phi(\vec p))\to ST_x(\psi(\vec p)))).$
\end{itemize}
where $\vec P$ are the unary predicate symbols corresponding to $\vec p$.
\end{prop}
\begin{proof}
By induction on the structure of $\phi$ and the satisfaction relation for each connective and variable, as well as the semantics of the first-order correspondence language. 
\end{proof}

\begin{remark}\label{aremark:simplify:translation}
In fact, since unary predicate symbols are interpreted only as admissible subsets (therefore regular open subsets), and the truth set of every formula $\phi$ is a regular open subset, there are some first-order formulas valid on all full possibility frames which are not first-order theorems (e.g.\ $ST_{x}(\phi)\leftrightarrow(\forall y\sqsubseteq x)(\exists z\sqsubseteq y)(\exists w\sqsupseteq z)ST_{w}(\phi)$). As a result, we have different ways to obtain translations. For example, if we obtain the translation of $\varphi\to\psi$ directly from the syntactic definition of $\to$, then we would have something different:
$$ST_x(\varphi_1\to\varphi_2):=(\forall y\sqsubseteq x)(\exists z\sqsubseteq y)(((\forall w\sqsubseteq z)\neg ST_w(\varphi_1))\lor ST_z(\varphi_2))$$
In fact, the translation given above is equivalent to the one in Definition \ref{adef:trans:2}: $$\llbracket(\forall y\sqsubseteq x)(\exists z\sqsubseteq y)(((\forall w\sqsubseteq z)\neg ST_w(\varphi_1))\lor ST_z(\varphi_2))\rrbracket^{\mathbb{F}, V}=\llbracket\neg\phi_1\lor\phi_2\rrbracket^{\mathbb{F}, V},$$ and 
$$\llbracket(\forall y\sqsubseteq x)(ST_y(\varphi_1)\to ST_y(\varphi_2))\rrbracket^{\mathbb{F}, V}=\llbracket\phi_1\to\phi_2\rrbracket^{\mathbb{F}, V},$$ which are the same.
\end{remark}

\subsection{Sahlqvist and inductive formulas}\label{asec:Sahlqvist}
In the present section, we will report on the definitions of sahlqvist and inductive inequalities for classical modal formulas. Sahlqvist formulas are very well known and their definition can be found e.g. in \cite[Definition 3.51]{BRV01}. Inductive formulas are defined by Goranko and Vakarelov in \cite{GorankoV06}. However, because in the following sections we are going to make crucial use of the order-theoretic properties of the algebraic interpretations of the logical connectives,in the present section we will specialize the general definition of adopted in unified correspondence to the signature of classical modal logic. Our presentation follows e.g. \cite{CPSZ,CPZ:Trans,PaSoZh16}.

\begin{definition}[Order-type of propositional variables]
We will consider the order-type of tuples of propositional variables $(p_1, \ldots, p_n)$, and say that $p_i$ has order-type 1 (resp.\ $\partial$) if $\epsilon_i=1$ (resp.\ $\epsilon_i=\partial$), and denote $\epsilon(p_i)=1$ (resp.\ $\epsilon(p_i)=\partial$) or $\epsilon(i)=1$ (resp.\ $\epsilon(i)=\partial$).
\end{definition}

\begin{definition}[Signed generation tree]\label{adef: signed gen tree}(cf.\ \cite[Definition 4]{CPZ:Trans})
The \emph{positive} (resp.\ \emph{negative}) {\em generation tree} of any given formula $\phi$ is defined by first labelling the root of the generation tree of $\phi$ with $+$ (resp.\ $-$), and then labelling the children nodes as follows:
\begin{itemize}
\item Assign the same sign to the children nodes of any node labelled with $\Box, \Diamond, \lor, \land$;
\item Assign the opposite sign to the child node of any node labelled with $\neg$;
\item Assign the opposite sign to the first child node and the same sign to the second child node of any node labelled with $\to$.
\end{itemize}
Nodes in signed generation trees are \emph{positive} (resp.\ \emph{negative}) if they are signed $+$ (resp.\ $-$).
\end{definition}

Signed generation trees will be used in the environment of inequalities $\phi\leq\psi$, where we will use the positive generation tree $+\phi$ and the negative one $-\psi$. We will also say that an inequality $\phi\leq\psi$ is \emph{uniform} in a variable $p$ if all occurrences of $p$ in $+\phi$ and $-\psi$ have the same sign, and that $\phi\leq\psi$ is $\epsilon$-\emph{uniform} in an array $\vec{p}$ if $\phi\leq\psi$ is uniform in $p$, occurring with the sign indicated by $\epsilon$, for each propositional variable $p$ in $\vec{p}$.

For any given formula $\phi(p_1,\ldots p_n)$, any order-type $\epsilon$ over $n$, and any $1 \leq i \leq n$, an \emph{$\epsilon$-critical node} in a signed generation tree of $\phi$ is a leaf node $+p_i$ when $\epsilon_i = 1$ or $-p_i$ when $\epsilon_i = \partial$. An $\epsilon$-{\em critical branch} in a signed generation tree is a branch from an $\epsilon$-critical node. The $\epsilon$-critical occurrences are intended to be those which the algorithm ALBA will solve for. We say that $+\phi$ (resp.\ $-\phi$) {\em agrees with} $\epsilon$, and write $\epsilon(+\phi)$ (resp.\ $\epsilon(-\phi)$), if every leaf node in the signed generation tree of $+\phi$ (resp.\ $-\phi$) is $\epsilon$-critical.

We will also use the notation $+\psi\prec \ast \phi$ (resp.\ $-\psi\prec \ast \phi$) to indicate that the subformula $\psi$ inherits the positive (resp.\ negative) sign from the signed generation tree $\ast \phi$, where $\ast\in\{+,-\}$. We will write $\epsilon(\gamma) \prec \ast \phi$ (resp.\ $\epsilon^\partial(\gamma) \prec \ast \phi$) to indicate that the signed generation subtree $\gamma$, with the sign inherited from $\ast \phi$, agrees with $\epsilon$ (resp.\ with $\epsilon^\partial$). We say that a propositional variable $p$ is \emph{positive} (resp.\ \emph{negative}) in $\phi$ if $+p\prec+\phi$ (resp.\ $-p\prec+\phi$).
\begin{definition}\label{adef:good:branches}(cf.\ \cite[Definition 5]{CPZ:Trans})
Nodes in signed generation trees are called \emph{$\Delta$-adjoints}, \emph{syntactically left residual (SLR)}, \emph{syntactically right adjoint (SRA)}, and \emph{syntactically right residual (SRR)}\footnote{For explanations of the terminologies here, we refer to \cite[Remark 3.24]{PaSoZh16}.}, according to Table \ref{aJoin:and:Meet:Friendly:Table}.

A branch in a signed generation tree is called a \emph{good branch} if it is the concatenation of two paths $P_1$ and $P_2$, one of which might be of length $0$, such that $P_1$ is a path from the leaf consisting (apart from variable nodes) of PIA-nodes only, and $P_2$ consists (apart from variable nodes) of Skeleton-nodes only. A good branch will be called \emph{excellent} if $P_1$ consists of SRA nodes only.
\begin{table}
\begin{center}
\begin{tabular}{| c | c |}
\hline
Skeleton &PIA\\
\hline
$\Delta$-adjoints & SRA \\
\begin{tabular}{ c c c c c c}
$+$ &$\vee$ &$\wedge$&\\
$-$ &$\wedge$ &$\vee$&\\
\end{tabular}
&
\begin{tabular}{c c c c }
$+$ &$\wedge$ &$\Box$ & $\neg$ \\
$-$ &$\vee$ &$\Diamond$ & $\neg$ \\
\end{tabular}\\
\hline
SLR &SRR\\
\begin{tabular}{c c c c c}
$+$ & $\wedge$ &$\Diamond$ & $\neg$ &\\
$-$ & $\vee$ &$\Box$ & $\neg$ & $\to$\\
\end{tabular}
&\begin{tabular}{c c c c}
$+$ &$\vee$ &$\to$\\
$-$ & $\wedge$ &\\
\end{tabular}
\\
\hline
\end{tabular}
\end{center}
\caption{Skeleton and PIA nodes.}\label{aJoin:and:Meet:Friendly:Table}
\vspace{-1em}
\end{table}
\end{definition}

\begin{definition}[Sahlqvist and inductive inequalities]\label{aInducive:Ineq:Def}(cf.\ \cite[Definition 6]{CPZ:Trans})
For any order-type $\epsilon$ and any irreflexive and transitive binary relation $<_\Omega$ on $p_1,\ldots p_n$, the signed generation tree $*\phi$ $(* \in \{-, + \})$ of a formula $\phi(p_1,\ldots p_n)$ is \emph{$(\Omega, \epsilon)$-inductive} if
\begin{enumerate}
\item for all $1 \leq i \leq n$, every $\epsilon$-critical branch with leaf $p_i$ is good (cf.\ Definition \ref{adef:good:branches});
\item every SRR-node in the critical branch is either $\bigstar(\gamma,\beta)$ or $\bigstar(\beta,\gamma)$, where the critical branch is in $\beta$, and
\begin{enumerate}
\item $\epsilon^\partial(\gamma) \prec \ast \phi$ (cf.\ the discussion before Definition \ref{adef:good:branches}), and
\item $p_k <_{\Omega} p_i$ for every $p_k$ that occurs in $\gamma$.
\end{enumerate}
\end{enumerate}

For any order-type $\epsilon$, the signed generation tree $*\phi$ of a formula $\phi(p_1,\ldots p_n)$ is \emph{$\epsilon$-Sahlqvist} if for all $1 \leq i \leq n$, every $\epsilon$-critical branch with leaf $p_i$ is excellent (cf.\ Definition \ref{adef:good:branches}).
					
We will refer to $<_{\Omega}$ as the \emph{dependence order} on the variables. An inequality $\phi\leq\psi$ is \emph{$(\Omega, \epsilon)$-Sahlqvist} (resp.\ \emph{$\epsilon$-inductive}) if the signed generation trees $+\phi$ and $-\psi$ are $(\Omega, \epsilon)$-Sahlqvist (resp.\ \emph{$\epsilon$-inductive}). An inequality $\phi\leq\psi$ is \emph{Sahlqvist} (resp.\ inductive) if it is $(\Omega, \epsilon)$-Sahlqvist (\emph{$\epsilon$}-inductive) for some ($\Omega$, $\epsilon$).
\end{definition}
\section{The algorithm ALBA for possibility semantics}\label{aSec:ALBA}
In the present section, we will give the algorithm ALBA for possibility semantics, which is similar to the version in \cite{CP:constructive,CoPa:non-dist}, especially when it comes to the approximation rules. ALBA receives an inequality $\phi \leq \psi$ as input and then proceeds in the following three stages:

The first stage is the preprocessing stage, which eliminates all uniformly occurring propositional variables, and exhaustively apply the distribution and splitting rules. This stage produces a finite number of inequalities, $\phi'_i \leq \psi'_i$, $1 \leq i \leq n$. Then the inequalities are rewritten as \emph{initial quasi-inequalities} $\bigamp S_i \Rightarrow \sf{Ineq}_i$, represented as pairs $(S_i, \sf{Ineq}_i)$ called \emph{systems}, where each $S_i$ is initialized to the empty set and each $\sf{Ineq}_i$ is initialized to $\phi'_i \leq \psi'_i$.

The second stage is the reduction and elimination stage, which applies the transformation rules to transform $(S_i, \mathsf{Ineq}_i)$ into a system which does not contain propositional variables, but only nominals. Such a system will be called \emph{pure}. Among all the rules, the Ackermann-rules eliminate the propositional variables, and the other rules rewrite the system to make the Ackermann rules applicable. When all propositional variables are eliminated, the algorithm produces pure quasi-inequalities $\bigamp S_i \Rightarrow \mathsf{Ineq}_i$.

The third stage is the output stage. If some of the systems cannot be purified, then the algorithm reports failure and terminate. Otherwise, the algorithm outputs the conjunction of the pure quasi-inequalities $\bigamp S_i \Rightarrow \mathsf{Ineq}_i$ as well as its regular open translation, which we denote by $\mathsf{ALBA}(\phi \leq \psi)$ and $\mathsf{FO}(\phi \leq \psi)$, respectively.

The three stages are reported on in detail below:

\paragraph{Stage 1: Preprocessing and initialization} ALBA receives an inequality $\phi \leq \psi$ as input.

In the signed generation tree of $+\phi$ and $-\psi$,
\begin{enumerate}
\item Apply the distribution rules:
\begin{enumerate}
\item Push down $+\Diamond, +\land$, $-\neg$ and $-\to$, by distributing them over nodes labelled with $+\lor$ which are Skeleton nodes, and
\item Push down $-\Box, -\lor$, $+\neg$ and $-\to$, by distributing them over nodes labelled with $-\land$ which are Skeleton nodes.
\end{enumerate}
\item Apply the monotone and antitone variable-elimination rules:
$$\infer{\alpha(\perp)\leq\beta(\perp)}{\alpha(p)\leq\beta(p)}
\qquad
\infer{\beta(\top)\leq\alpha(\top)}{\beta(p)\leq\alpha(p)}
$$
where $\beta(p)$ is positive in $p$ and $\alpha(p)$ is negative in $p$.
\item Apply the splitting rules:
$$\infer{\alpha\leq\beta\ \ \ \alpha\leq\gamma}{\alpha\leq\beta\land\gamma}
\qquad
\infer{\alpha\leq\gamma\ \ \ \beta\leq\gamma}{\alpha\lor\beta\leq\gamma}
$$
\end{enumerate}
After applying the rules above exhaustively, ALBA produces a set of inequalities $\{\phi_i' \leq \psi_i'\mid 1\leq i\leq n\}$. Then the inequalities are rewritten as \emph{initial quasi-inequalities} $\bigamp S_i \Rightarrow \sf{Ineq}_i$, represented as pairs $(S_i, \sf{Ineq}_i)$ called \emph{systems}, where each $S_i$ is initialized to the empty set and each $\sf{Ineq}_i$ is initialized to $\phi'_i \leq \psi'_i$. Each initial system is processed separately in Stage 2, where we will omit indices $i$.

\paragraph{Stage 2: Reduction and elimination}
This stage aims at eliminating all propositional variables from a system $(S, \mathsf{Ineq})$. This is done by the following \emph{approximation rules}, \emph{residuation rules}, \emph{splitting rules}, and \emph{Ackermann-rules}, collectively called \emph{reduction rules}. The formulas and inequalities in this subsection are from the expanded modal language with nominals and black connectives.

\paragraph{Approximation rules.} There are four approximation rules. Each of these rules simplifies $\mathsf{Ineq}$ and adds an inequality to $S$. The notation $\phi(!x)$ indicate that $x$ occurs only once in $\phi$, and the branch of $\phi(!x)$ starting at $x$ means the path from $x$ to the root.
\begin{description}
\item[Left-positive approximation rule.] $\phantom{a}$
\begin{center}
\AxiomC{$(S, \;\; \phi'(\gamma / !x)\leq \psi')$}
\RightLabel{$(L^+A)$}
\UnaryInfC{$(S\! \cup\! \{ \nomi \leq \gamma\},\;\; \phi'(\nomi / !x)\leq \psi')$}
\DisplayProof
\end{center}
where $+x \prec +\phi'(!x)$, the branch of $+\phi'(!x)$ starting at $+x$ is SLR, $\gamma$ belongs to the basic modal language and $\nomi$ is a nominal variable not occurring in $S$ or $\phi'(\gamma / !x)\leq \psi'$.
\item[Left-negative approximation rule.]$\phantom{a}$
\begin{center}
\AxiomC{$(S, \;\; \phi'(\gamma / !x)\leq \psi')$}
\RightLabel{$(L^-A)$}
\UnaryInfC{$(S\! \cup\! \{ \gamma\leq\neg\nomi\},\;\; \phi'(\neg\nomi/ !x)\leq \psi')$}
\DisplayProof
\end{center}
where $-x \prec +\phi'(!x)$, the branch of $+\phi'(!x)$ starting at $-x$ is SLR, $\gamma$ belongs to the basic modal language and $\nomi$ is a nominal variable not occurring in $S$ or $\phi'(\gamma / !x)\leq \psi'$.
\item[Right-positive approximation rule.]$\phantom{a}$
\begin{center}
\AxiomC{$(S, \;\; \phi'\leq \psi'(\gamma / !x))$}
\RightLabel{$(R^+A)$}
\UnaryInfC{$(S\! \cup\! \{ \nomi \leq \gamma\},\;\; \phi'\leq \psi'(\nomi/ !x))$}
\DisplayProof
\end{center}
where $+x \prec -\psi'(!x)$, the branch of $-\psi'(!x)$ starting at $+x$ is SLR, $\gamma$ belongs to the basic modal language and $\nomi$ is a nominal variable not occurring in $S$ or $\phi'\leq \psi'(\gamma / !x)$.
\item[Right-negative approximation rule.] $\phantom{a}$
\begin{center}
\AxiomC{$(S, \;\; \phi'\leq \psi'(\gamma / !x))$}
\RightLabel{$(R^-A)$}
\UnaryInfC{$(S\! \cup\! \{ \gamma\leq \neg\nomi\},\;\; \phi'\leq \psi'(\neg\nomi/ !x))$}
\DisplayProof
\end{center}
where $-x \prec -\psi'(!x)$, the branch of $-\psi'(!x)$ starting at $-x$ is SLR, $\gamma$ belongs to the basic modal language and $\nomi$ is a nominal variable not occurring in $S$ or $\phi'\leq \psi'(\gamma / !x))$.
\end{description}

We will restrict the applications of approximation rules to nodes $!x$ giving rise to {\em maximal} SLR branches, i.e.\ branches that cannot be extended further. Such applications will be called {\em pivotal}. Also, executions of ALBA in which approximation rules are applied only pivotally will be referred to as {\em pivotal}.

\paragraph{Residuation rules.} These rules rewrite a chosen inequality in $S$ into another inequality:

\begin{prooftree}
\AxiomC{$\Diamond\gamma\leq\delta$}
\RightLabel{($\Diamond$-Res)}
\UnaryInfC{$\gamma\leq\blacksquare\delta$}
\AxiomC{$\neg\gamma\leq\delta$}
\RightLabel{($\neg$-Res-Left)}
\UnaryInfC{$\neg\delta\leq\gamma$}
\noLine\BinaryInfC{}
\end{prooftree}

\begin{prooftree}
\AxiomC{$\gamma\leq\Box\delta$}
\RightLabel{($\Box$-Res)}
\UnaryInfC{$\Diamondblack\gamma\leq\delta$}
\AxiomC{$\gamma\leq\neg\delta$}
\RightLabel{($\neg$-Res-Right)}
\UnaryInfC{$\delta\leq\neg\gamma$}
\noLine\BinaryInfC{}
\end{prooftree}

\begin{prooftree}
\AxiomC{$\gamma\land\delta\leq\beta$}
\RightLabel{($\land$-Res-1)}
\UnaryInfC{$\gamma\leq\delta\to\beta$}
\AxiomC{$\gamma\leq\delta\lor\beta$}
\RightLabel{($\lor$-Res-1)}
\UnaryInfC{$\gamma\land\neg\delta\leq\beta$}
\noLine\BinaryInfC{}
\end{prooftree}

\begin{prooftree}
\AxiomC{$\gamma\land\delta\leq\beta$}
\RightLabel{($\land$-Res-2)}
\UnaryInfC{$\delta\leq\gamma\to\beta$}
\AxiomC{$\gamma\leq\delta\lor\beta$}
\RightLabel{($\lor$-Res-2)}
\UnaryInfC{$\gamma\land\neg\beta\leq\delta$}
\noLine\BinaryInfC{}
\end{prooftree}

\begin{prooftree}
\AxiomC{$\gamma\leq\delta\to\beta$}
\RightLabel{($\to$-Res-1)}
\UnaryInfC{$\gamma\land\delta\leq\beta$}
\AxiomC{$\gamma\leq\delta\to\beta$}
\RightLabel{($\to$-Res-2)}
\UnaryInfC{$\delta\leq\gamma\to\beta$}
\noLine\BinaryInfC{}
\end{prooftree}

\paragraph{Right Ackermann-rule.} $\phantom{a}$
\begin{center}
\AxiomC{$(\{ \alpha_i \leq p \mid 1 \leq i \leq n \} \cup \{ \beta_j(p)\leq \gamma_j(p) \mid 1 \leq j \leq m \}, \;\; \mathsf{Ineq})$}
\RightLabel{$(RAR)$}
\UnaryInfC{$(\{ \beta_j(\bigvee_{i=1}^n \alpha_i)\leq \gamma_j(\bigvee_{i=1}^n \alpha_i) \mid 1 \leq j \leq m \},\;\; \mathsf{Ineq})$}
\DisplayProof
\end{center}
where:
\begin{itemize}
\item $p$ does not occur in $\alpha_1, \ldots, \alpha_n$ or in $\mathsf{Ineq}$,
\item $\beta_{1}(p), \ldots, \beta_{m}(p)$ are positive in $p$, and
\item $\gamma_{1}(p), \ldots, \gamma_{m}(p)$ are negative in $p$.
\end{itemize}
\paragraph{Left Ackermann-rule.}$\phantom{a}$
\begin{center}
\AxiomC{$(\{ p \leq \alpha_i \mid 1 \leq i \leq n \} \cup \{ \beta_j(p)\leq \gamma_j(p) \mid 1 \leq j \leq m \}, \;\; \mathsf{Ineq})$}
\RightLabel{$(LAR)$}
\UnaryInfC{$(\{ \beta_j(\bigwedge_{i=1}^n \alpha_i)\leq \gamma_j(\bigwedge_{i=1}^n \alpha_i) \mid 1 \leq j \leq m \},\;\; \mathsf{Ineq})$}
\DisplayProof
\end{center}
where:
\begin{itemize}
\item $p$ does not occur in $\alpha_1, \ldots, \alpha_n$ or in $\mathsf{Ineq}$,
\item $\beta_{1}(p), \ldots, \beta_{m}(p)$ are negative in $p$, and
\item $\gamma_{1}(p), \ldots, \gamma_{m}(p)$ are positive in $p$.
\end{itemize}
\paragraph{Stage 3: Success, failure and output}

If Stage 2 succeeded in eliminating all propositional variables from each system, then for each $i$, the system becomes $(S_i, \mathsf{Ineq}_i)$, where 
$S_i=\{\phi_{i_k}(\vec \nomj_i)\leq\psi_{i_k}(\vec \nomj_i)\}_k$, $\mathsf{Ineq}_i$ is $\phi_{i}(\vec \nomj_i)\leq\psi_{i}(\vec \nomj_i)$. The algorithm returns the conjunction of these purified quasi-inequalities and its regular open translation
$$\bigwedge_i\forall\vec j_i(\bigwedge_k(\forall x(ST_x(\phi_{i_k}(\vec \nomj_i))\to ST_x(\psi_{i_k}(\vec \nomj_i))))\to\forall x(ST_x(\phi_{i}(\vec \nomj_i))\to ST_x(\psi_{i}(\vec \nomj_i)))),$$
denoted by $\mathsf{ALBA}(\phi \leq \psi)$ and $\mathsf{FO}(\phi \leq \psi)$, respectively. Otherwise, the algorithm reports failure and terminates.

\begin{example}
Let us consider the following inequality $\Box p\leq p$. 

In the first stage, no rules are applied, so the initial system is $$(\emptyset, \Box p\leq p).$$

In the second stage, by applying the left-positive approximation rule, we get

$$(\{\nomi\leq \Box p\}, \nomi\leq p);$$

then by the right-negative approximation rule, we get 

$$(\{\nomi\leq \Box p, p\leq\neg\nomj\}, \nomi\leq \neg\nomj);$$

then by the residuation rule, we get 

$$(\{\Diamondblack\nomi\leq p, p\leq\neg\nomj\}, \nomi\leq \neg\nomj);$$

then by the right Ackermann rule, we get

$$(\{\Diamondblack\nomi\leq\neg\nomj\}, \nomi\leq \neg\nomj);$$

this system is equivalent to the following pure-inequality:

$$\nomi\leq\Diamondblack\nomi.$$

In the third stage, the first-order correspondent is given as follows:

\begin{center}
$\forall i\forall x(ST_{x}(\nomi)\to ST_{x}(\Diamondblack\nomi))$

$\forall i\forall x(\mathsf{RO}_{x}(x=i)\to ST_{x}(\Diamondblack\nomi))$

$\forall i\forall x(\mathsf{RO}_{x}(x=i)\to \mathsf{RO}_{x}(\exists y (Ryx\land\mathsf{RO}_{y}(y=i)))).$
\end{center}
\end{example}

As we can see from this example, the first-order correspondent of modal formulas will become much more complicated even for very simple input.
\section{Soundness}\label{aSec:soundness}

In the present section, we will prove the soundness of the algorithm with respect to regular open dual BAOs of full possibility frames. For most of the rules, the proof is similar to existing settings \cite{CP:constructive,CoPa:non-dist}, and hence omitted. We will focus on the approximation rules, which are the only rules dealing with the variant interpretations of nominals.
\begin{theorem}[Soundness]\label{aCrctns:Theorem} If ALBA succeeds on an input inequality $\phi\leq\psi$ and outputs $\mathsf{FO}(\phi\leq\psi)$, then for any full possibility frame $\mathbb{F}_{\mathsf{RO}}$, $$\mathbb{F}_{\mathsf{RO}}\vDash\phi \leq \psi\mbox{ iff }\mathbb{F}_{\mathsf{RO}}\vDash\mathsf{FO}(\phi \leq \psi).$$
\end{theorem}
\begin{proof}
The proof goes similarly to \cite[Theorem 5.1]{CP:constructive}. Let $\phi_i\leq\psi_i$, $1\leq i\leq n$ denote the inequalities produced by preprocessing $\phi\leq\psi$ after Stage 1, and $(S_{i},\mathsf{Ineq}_{i})$, $1\leq i\leq n$ denote the corresponding quasi-inequalities produced by ALBA after Stage 2. It suffices to show the equivalence from (\ref{aCrct:Eqn0}) to (\ref{aCrct:Eqn6}) given below:
\begin{eqnarray}
&&\mathbb{F}_{\mathsf{RO}}\vDash\phi\leq\psi\label{aCrct:Eqn0}\\
&&\mathbb{B}_{\mathsf{RO}}\vDash\phi\leq\psi\label{aCrct:Eqn1}\\
&&\mathbb{B}_{\mathsf{RO}}\vDash\phi_i\leq\psi_i,\mbox{ for all }1\leq i\leq n\label{aCrct:Eqn2}\\
&&\mathbb{B}_{\mathsf{RO}}\vDash\bigamp\varnothing\Rightarrow\phi_i\leq\psi_i,\mbox{ for all }1\leq i\leq n\label{aCrct:Eqn3}\\
&&\mathbb{B}_{\mathsf{RO}}\vDash\bigamp S_i\Rightarrow\mathsf{Ineq}_i,\mbox{ for all }1\leq i\leq n\label{aCrct:Eqn4}\\
&&\mathbb{B}_{\mathsf{RO}}\vDash\mathsf{ALBA}(\phi\leq\psi)\label{aCrct:Eqn5}\\
&&\mathbb{F}_{\mathsf{RO}}\vDash\mathsf{FO}(\phi\leq\psi)\label{aCrct:Eqn6}
\end{eqnarray}
\begin{itemize}
\item The equivalence between (\ref{aCrct:Eqn0}) and (\ref{aCrct:Eqn1}) follows from Proposition \ref{aprop:equiv:duality:expanded};
\item to show the equivalence of (\ref{aCrct:Eqn1}) and (\ref{aCrct:Eqn2}), it suffices to show the soundness of the rules in Stage 1, which can be proved in the same way as in \cite[Theorem 5.1]{CP:constructive};
\item the equivalence between (\ref{aCrct:Eqn2}) and (\ref{aCrct:Eqn3}) is immediate;
\item the equivalence between (\ref{aCrct:Eqn3}) and (\ref{aCrct:Eqn4}) follows from Propositions \ref{aprop:soundness:left:positive}, \ref{aprop:soundness:approximation}, \ref{aprop:soundness:other:rules} and \ref{aprop:soundness:Ackermann} below;
\item the equivalence between (\ref{aCrct:Eqn4}) and (\ref{aCrct:Eqn5}) is again immediate;
\item the equivalence between (\ref{aCrct:Eqn5}) and (\ref{aCrct:Eqn6}) follows from Propositions \ref{aprop:equiv:duality:expanded} and \ref{aprop:translation}.
\end{itemize}
\end{proof}

Therefore, it remains to show the soundness of Stage 2, for which it suffices to show the soundness of the approximation rules, the residuation rules, the Ackermann rules and the splitting rules. For the residuation rules, the only property needed is the adjunction property, and the proofs are similar to existing settings (see e.g.\ \cite[Lemma 8.4]{CoPa12}), and hence are omitted. For the splitting rules, their soundness are immediate. For the Ackermann rules, since the nominals and propositional variables are interpreted as elements in $\mathbb{B}_{\mathsf{RO}}$, and the basic and black connectives are all interpreted as operations from (products of) $\mathbb{B}_{\mathsf{RO}}$ to $\mathbb{B}_{\mathsf{RO}}$, the ``minimal valuation'' formula is always interpreted as an element in $\mathbb{B}_{\mathsf{RO}}$, therefore the soundness proof of the Ackermann lemmas are the same as \cite[Lemma 6.3, 6.4]{CoPa:non-dist}. The only rules which need special attention are the approximation rules. Since they are order-dual to each other, it suffices to focus on one of these rules.

Let us consider the left-positive approximation rule. Consider a system $(S, \phi(\alpha, \gamma_1, \ldots, \gamma_n)\leq\psi)$ where $\alpha$ is the formula to be approximated, and the regular open dual BAO $\mathbb{B}_{\mathsf{RO}}$ and assignment $\theta$ where the system is interpreted. By Proposition \ref{aprop:regular:open:join}, $\alpha^{(\mathbb{B}_{\mathsf{RO}},\theta)}=\bigvee_{\mathsf{RO}}\{x\in\mathsf{PsAt}(\mathbb{B}_{\mathsf{RO}})\mid x\leq\alpha^{(\mathbb{B}_{\mathsf{RO}},\theta)}\}$.

The soundness of the left-positive approximation rule is justified by the following proposition:

\begin{prop}
Given a regular open dual BAO $\mathbb{B}_{\mathsf{RO}}$ and an assignment $\theta$ on $\mathbb{B}_{\mathsf{RO}}$, $$(\mathbb{B}_{\mathsf{RO}}, \theta)\vDash\bigamp S\Rightarrow \phi(\alpha, \gamma_1, \ldots, \gamma_n)\leq\psi$$
iff for any $x\in\mathsf{PsAt}(\mathbb{B}_{\mathsf{RO}})$, 
$$(\mathbb{B}_{\mathsf{RO}}, \theta'_{x})\vDash\nomj\leq\alpha\ \&\bigamp S\Rightarrow\phi(\nomj, \gamma_1, \ldots, \gamma_n)\leq\psi,$$
where $\theta'_{x}$ is the $\nomj$-variant of $\theta$ such that $\theta'_{x}(\nomj)=x$ and $\theta'_x$ agrees with $\theta$ on other variables.
\end{prop}
\begin{proof}

First of all, by the requirement of the approximation rule, we have that $\phi(\underline{\hspace{2.5mm}}, \gamma_1, \ldots, \gamma_n)$ is completely join-preserving in the empty coordinate. By Proposition \ref{aprop:regular:open:join}, $\alpha^{(\mathbb{B}_{\mathsf{RO}},\theta)}=\bigvee_{\mathsf{RO}}\{x\in\mathsf{PsAt}(\mathbb{B}_{\mathsf{RO}})\mid x\leq\alpha^{(\mathbb{B}_{\mathsf{RO}},\theta)}\}$. By complete distributivity, we have\footnote{For simplicity of notation, we omit the superscript $(\mathbb{B}_{\mathsf{RO}},\theta)$ of formulas and inequalities except for the last but one line; indeed, in the last two lines, the interpretation is in $(\mathbb{B}_{\mathsf{RO}},\theta'_x)$, and in other lines, the interpretation is in $(\mathbb{B}_{\mathsf{RO}},\theta)$.}:
\begin{center}
\begin{tabular}{r l}
& $(\mathbb{B}_{\mathsf{RO}},\theta)\vDash(\bigamp S)^{}\Rightarrow \phi(\alpha, \gamma_1, \ldots, \gamma_n)\leq\psi$\\
iff & $(\bigamp S)\Rightarrow \phi(\bigvee_{\mathsf{RO}}\{x\in\mathsf{PsAt}(\mathbb{B}_{\mathsf{RO}})\mid x\leq\alpha\}, \gamma_1, \ldots, \gamma_n)\leq\psi$\\
iff & $(\bigamp S)\Rightarrow \bigvee_{\mathsf{RO}}\{\phi(x, \gamma_1, \ldots, \gamma_n)\mid x\in\mathsf{PsAt}(\mathbb{B}_{\mathsf{RO}})\mbox{ and }x\leq\alpha\}\leq\psi$\\
iff & $(\bigamp S)\Rightarrow (\forall x\in\mathsf{PsAt}(\mathbb{B}_{\mathsf{RO}})\mbox{ s.t. }x\leq\alpha)(\phi(x, \gamma_1, \ldots, \gamma_n)\leq\psi)$\\
iff & $(\forall x\in\mathsf{PsAt}(\mathbb{B}_{\mathsf{RO}})\mbox{ s.t. }x\leq\alpha)(\bigamp S\Rightarrow\phi(x, \gamma_1, \ldots, \gamma_n)\leq\psi)$\\
iff & $(\forall x\in\mathsf{PsAt}(\mathbb{B}_{\mathsf{RO}}))(x\leq\alpha\&\bigamp S\Rightarrow\phi(x, \gamma_1, \ldots, \gamma_n)\leq\psi)$\\
iff & $(\forall x\in\mathsf{PsAt}(\mathbb{B}_{\mathsf{RO}}))(x\leq\alpha^{(\mathbb{B}_{\mathsf{RO}},\theta'_{x})}\&(\bigamp S)^{(\mathbb{B}_{\mathsf{RO}},\theta'_{x})}\Rightarrow\phi^{(\mathbb{B}_{\mathsf{RO}},\theta'_{x})}(x, \gamma_1^{(\mathbb{B}_{\mathsf{RO}},\theta'_{x})}, \ldots, \gamma_n^{(\mathbb{B}_{\mathsf{RO}},\theta'_{x})})\leq\psi^{(\mathbb{B}_{\mathsf{RO}},\theta'_{x})})$\\
iff & $(\forall x\in\mathsf{PsAt}(\mathbb{B}_{\mathsf{RO}}))((\mathbb{B}_{\mathsf{RO}}, \theta'_{x})\vDash(\nomj\leq\alpha\&\bigamp S\Rightarrow\phi(\nomj, \gamma_1, \ldots, \gamma_n)\leq\psi)).$
\end{tabular}
\end{center}
Notice that $\nomj$ does not occur in $S,\alpha,\phi,\psi,\gamma_1,\ldots,\gamma_n$.
\end{proof}
Notice again that here we interpret the nominals as regular opens closures of atoms in the full dual BAO rather than atoms or complete join-irreducibles since they might not be available in $\mathbb{B}_{\mathsf{RO}}$.

By the proposition above, we have the soundness of the left-positive approximation rule as an easy corollary:
\begin{prop}\label{aprop:soundness:left:positive}
The left-positive approximation rule is sound.
\end{prop}
By order-dual arguments, the other three approximation rules are sound:
\begin{prop}\label{aprop:soundness:approximation}
The left-negative approximation rule, the right-positive approximation rule and the right-negative approximation rule are sound.
\end{prop}
Following from standard soundness proof of ALBA \cite[Lemma 8.3, 8.4]{CoPa12}, we have the following:
\begin{prop}\label{aprop:soundness:other:rules}
The distribution rules, the splitting rules, the monotone and antitone rules and the residuation rules are sound.
\end{prop}

For the Ackermann rules, the proof is also similar. It is worth mentioning that all the formulas in the expanded modal language are interpreted in $\mathbb{B}_{\mathsf{RO}}$, as well as that we are working at the discrete duality level, therefore the topological Ackermann lemmas (cf.\ \cite[Lemma 9.3, 9.4]{CoPa12}) are not needed in the soundness proof here. 
\begin{prop}\label{aprop:soundness:Ackermann}
The Ackermann rules are sound.
\end{prop}

\section{Success}\label{aSec:success}
The proof that ALBA succeeds on inductive inequalities goes similarly to \cite{CoPa:non-dist}, therefore in what follows we will only sketch its main lines without giving details, and refer the reader to \cite{CoPa:non-dist} for an exhaustive treatment.

The following definition is meant to describe the shape of the inequality after Stage 1:

\begin{definition}
An $(\Omega, \epsilon)$-inductive inequality is called {\em definite} if all Skeleton nodes are SLR nodes in its critical branches.
\end{definition}

\begin{lemma}\label{aPre:Process:Lemma}
Let $\{\phi_i\leq\psi_i\}$ be the set of inequalities obtained after preprocessing an $(\Omega, \epsilon)$-inductive inequality $\phi\leq\psi$. Then each $\phi_i\leq\psi_i$ is a definite $(\Omega,\epsilon)$-inductive inequality.
\end{lemma}

The following definition describes the shape of a system after approximation rules have been exhaustively applied pivotally until no propositional variables remain in $\mathsf{Ineq}$:

\begin{definition}\label{aStripped:Def}
A system $(S, \mathsf{Ineq})$ is called \emph{$(\Omega, \epsilon)$-stripped} if $\mathsf{Ineq}$ is pure and for each $\xi \leq \chi \in S$, the following conditions hold:
\begin{enumerate}
\item one of $+\chi$ and $-\xi$ is pure, and the other is $(\Omega, \epsilon)$-inductive;
\item every $\epsilon$-critical branch in $+\chi$ and $-\xi$ is PIA.
\end{enumerate}
\end{definition}

\begin{lemma}\label{aStripping:Lemma}
For any definite $(\Omega, \epsilon)$-inductive inequality $\phi \leq \psi$, the system $(\emptyset, \phi \leq \psi)$ can be rewritten into an $(\Omega, \epsilon)$-stripped system by the application of approximation rules.
\end{lemma}
The next definition describes the shape of a system after exhaustively applying reduction rules other than the Ackermann rules:
\begin{definition}\label{aAckermann:Ready:Def}
An $(\Omega, \epsilon)$-stripped system $(S, \mathsf{Ineq})$ is \emph{Ackermann-ready} with respect to $p_i$ with $\epsilon_i = 1$ (resp.\ $\epsilon_i = \partial$) if every $\xi \leq \chi \in S$ has one of the following forms:
\begin{enumerate}
\item $\xi \leq p_i$ with $\xi$ pure (resp.\ $p_i \leq \chi$ with $\chi$ pure), or
\item $\xi \leq \chi$ where $- \xi$ or $+ \chi$ are $\epsilon^{\partial}$-uniform in $p_i$.
\end{enumerate}
\end{definition}
Notice that if a system is Ackermann-ready with respect to $p_i$, then the right or left Ackermann-rule (depending on whether $\epsilon_i = 1$ or $\epsilon_i = \partial$) is applicable.

The next lemma states that an $(\Omega, \epsilon)$-stripped system can be equivalently transformed into one in which the $\Omega$-minimal variables can be eliminated:
\begin{lemma}\label{aAckermann:Ready:LEmma}
If $(S, \mathsf{Ineq})$ is $(\Omega, \epsilon)$-stripped and $p_i$ is $\Omega$-minimal among the remaining propositional variables occurring in $(S, \mathsf{Ineq})$, then $(S, \mathsf{Ineq})$ can be rewritten by the application of residuation and splitting rules into a system which is \emph{Ackermann-ready} with respect to $p_i$.
\end{lemma}
The next lemma guarantees that the process of eliminating one propositional variable can be repeated as needed:
\begin{lemma}\label{aApp:Of:Ackermann:Lemma}
After applying the Ackermann rule to an $(\Omega, \epsilon)$-stripped system which is Ackermann-ready with respect to $p_i$, a system is obtained which is again $(\Omega, \epsilon)$-stripped.\end{lemma}
Therefore, all propositional variables can be eliminated from an  $(\Omega, \epsilon)$-stripped system, by repeatedly applying Ackermann rules. The discussion above outlines the proof of the following:
\begin{theorem}
$\mathrm{ALBA}$ succeeds on all inductive inequalities.
\end{theorem}

\section{Canonicity}\label{aSec:Canonicity}
In the present section, we prove that inductive inequalities are filter-canonical. Our proof is alternative to the one given by Holliday \cite[Theorem 7.20]{Ho16}. Indeed, Holliday's proof follows from the constructive canonicity of inductive inequalities proved in \cite{CP:constructive}, observing that the filter completions of BAOs coincide with their constructive canonical extensions (cf.\ \cite[Theorem 5.46]{Ho16}). In this section, we make use of the possibility semantics counterpart of the canonicity-via-correspondence argument, which is a variation of the standard U-shaped argument represented in the diagram below. For this U-shaped argument to be supported, we need to prove that ALBA is sound also with respect to the filter-descriptive possibility frames (see Definition \ref{adef:filter:desctiptive:frame}), i.e.\ the dual relational structure of Boolean algebras with operators in the framework of possibility semantics. Hence, in what follows, we will mainly focus on this aspect.

\begin{center}
\begin{tabular}{l c l}\label{atable:U:shape:descriptive}
$\mathbb{F}_{\mathsf{FD}}\vDash\phi(\vec p)\leq\psi(\vec p)$ & &$\mathbb{F}_{\mathsf{RO}}\vDash\phi(\vec p)\leq\psi(\vec p)$\\
\ \ \ \ \ \ $\Updownarrow$ & &\ \ \ \ \ \ $\Updownarrow$\\
$\mathbb{B}_{\mathsf{RO}}\vDash_{\mathsf{FD}}\phi(\vec p)\leq\psi(\vec p)$ & &$\mathbb{B}_{\mathsf{RO}}\vDash\phi(\vec p)\leq\psi(\vec p)$\\
\ \ \ \ \ \ $\Updownarrow$ & &\ \ \ \ \ \ $\Updownarrow$\\
$\mathbb{B}_{\mathsf{RO}}\vDash_{\mathsf{FD}}(\forall \vec{p}\in \mathbb{B}_{\mathsf{FD}})(\phi(\vec p)\leq\psi(\vec p))$ & & $\mathbb{B}_{\mathsf{RO}}\vDash(\forall \vec{p}\in \mathbb{B}_{\mathsf{RO}})(\phi(\vec p)\leq\psi(\vec p))$\\
\ \ \ \ \ \ $\Updownarrow$ & &\ \ \ \ \ \ $\Updownarrow$\\
$\mathbb{B}_{\mathsf{RO}}\vDash_{\mathsf{FD}}(\forall\vec{\mathbf{i}}\in \mathsf{PsAt}(\mathbb{B}_{\mathsf{RO}}))$Pure$(\phi(\vec p)\leq\psi(\vec p))$ &\ \ \ \ \ \ &$\mathbb{B}_{\mathsf{RO}}\vDash(\forall\vec{\mathbf{i}}\in \mathsf{PsAt}(\mathbb{B}_{\mathsf{RO}}))$Pure$(\phi(\vec p)\leq\psi(\vec p))$\\
\ \ \ \ \ \ $\Updownarrow$ & &\ \ \ \ \ \ $\Updownarrow$\\
$\mathbb{F}_{\mathsf{FD}}\vDash$FO(Pure$(\phi(\vec p)\leq\psi(\vec p))$) &\ \ \ $\Leftrightarrow$ \ \ \ &$\mathbb{F}_{\mathsf{RO}}\vDash$FO(Pure$(\phi(\vec p)\leq\psi(\vec p))$)
\end{tabular}
\end{center}

Here $\mathbb{B}_{\mathsf{FD}}$ denote the dual BAO of the filter-descriptive possibility frame $\mathbb{F}_{\mathsf{FD}}$, $\mathbb{F}_{\mathsf{RO}}$ is the underlying full possibility frame of $\mathbb{F}_{\mathsf{FD}}$, and $\mathbb{B}_{\mathsf{RO}}$ is the dual BAO of $\mathbb{F}_{\mathsf{RO}}$. 

We will first provide the semantic environment of the present section, i.e.\ filter-descriptive possibility frames and canonical extensions in Section \ref{asec:filter:descriptive}. Section \ref{asec:soundness:topological} gives the soundness proof of the algorithm with respect to (the dual BAOs of) filter-descriptive possibility frames, where the topological Ackermann lemmas are given. The correspondence and canonicity results are collected in Section \ref{aSec:results:topo}.

\subsection{Filter-descriptive possibility frames and canonical extensions}\label{asec:filter:descriptive}

In the present section, we will collect basic definitions of filter-descriptive possibility frames as well as canonical extensions. For more details, the reader is referred to \cite{Ho16} and \cite{BeBlWo06}.

\subsubsection{Filter-descriptive possibility frames}
In \cite{Ho16}, Holliday introduced filter-descriptive possibility frames, i.e.\ the possibility semantics counterpart of descriptive general (Kripke) frames \cite[Section 5.5]{BRV01}, in which restrictions on the admissible set are imposed. The following definition is one of these restrictions.
\begin{definition}[Tightness](cf.\ \cite[Definition 4.31]{Ho16})
A possibility frame $\mathbb{F}=(W, \sqsubseteq, R, \mathsf{P})$ is said to be 
\begin{itemize}
\item $R$-\emph{tight}, if $(\forall w, v\in W)(\forall X\in\mathsf{P}(w\in \Box_{P}(X) \Rightarrow v\in X)\Rightarrow wRv)$;
\item $\sqsubseteq$-\emph{tight}, if $(\forall w, v\in W)(\forall X\in\mathsf{P}(w\in X \Rightarrow v\in X)\Rightarrow v\sqsubseteq w)$;
\item \emph{tight}, if it is both $R$-tight and $\sqsubseteq$-tight.
\end{itemize}
\end{definition}
The filter-descriptive possibility frames are introduced as below:
\begin{definition}[Filter-descriptive possibility frames] (cf.\ \cite[Definition 5.39]{Ho16})\label{adef:filter:desctiptive:frame}
A possibility frame $\mathbb{F}=(W, \sqsubseteq, R, \mathsf{P})$ is said to be \emph{filter-descriptive}, if the following conditions hold:
\begin{itemize}
\item it is tight;
\item for every proper filter\footnote{A \emph{proper filter} in a BAO $\mathbb{B}$ is a non-empty subset $F\subseteq B$ such that 
\begin{itemize}
\item for any $a, b\in F$, $a\land b\in F$;
\item for any $a, b\in B$, if $a\leq b$ and $a\in B$, then $b\in B$;
\item $\bot\notin F$.
\end{itemize}} $F$ in $\mathbb{B}_{\mathsf{P}}$, there exists an element $w\in W$ such that $F=P(w)=\{X\in \mathsf{P}\mid w\in X\}$.
\end{itemize}
\end{definition}

\begin{remark}\label{aremark:underlying:full} As we mentioned on page \pageref{aP:vs:RO}, for any possibility frame $\mathbb{F}=(W, \sqsubseteq, R, \mathsf{P})$, $\mathsf{RO}(W,\tau_{\sqsubseteq})$ might not be closed under the three conditions in Definition \ref{adef:poss:frame:model}. However, if $\mathbb{F}$ is a filter-descriptive possibility frame, $\mathsf{RO}(W,\tau_{\sqsubseteq})$ is always closed under the three conditions, i.e.\ given any filter-descriptive possibility frame $\mathbb{F}_{\mathsf{FD}}=(W, \sqsubseteq, R, \mathsf{P})$, its underlying full possibility frame $\mathbb{F}_{\mathsf{RO}}=(W, \sqsubseteq, R, \mathsf{RO}(W, \tau_{\sqsubseteq}))$ is well-defined (see \cite[Theorem 5.32.1, Proposition 5.40]{Ho16}).
\end{remark}
In this section, we will mainly work on the dual algebraic side, i.e.\ the dual BAOs of filter-descriptive possibility frames and the dual BAOs of their underlying full possibility frames. Indeed, the latter are the \emph{constructive canonical extension} of the former, which we will discuss below.

\subsubsection{Canonical extension}

As is known in the setting of Kripke semantics and its algebraic counterpart, given a descriptive general frame $\mathbb{G}$ and its underlying Kripke frame $\mathbb{F}$, the dual BAO of $\mathbb{F}$ is the canonical extension of the dual BAO of $\mathbb{G}$. It is natural to ask what is the relation between the dual BAO of a filter-descriptive possibility frame $\mathbb{F}_{\mathsf{FD}}$ and the dual BAO of its underlying full possibility frame $\mathbb{F}_{\mathsf{RO}}$. Indeed, by \cite[Theorem 5.46]{Ho16}, given any filter-descriptive possibility frame $\mathbb{F}_{\mathsf{FD}}=(W, \sqsubseteq, R, \mathsf{P})$ and its underlying full possibility frame $\mathbb{F}_{\mathsf{RO}}=(W, \sqsubseteq, R, \mathsf{RO}(W, \tau_{\sqsubseteq}))$, the dual BAO of $\mathbb{F}_{\mathsf{RO}}$ is the \emph{constructive canonical extension} of $\mathbb{F}_{\mathsf{FD}}$. In what follows, we will collect the basic definitions about canonical extensions. We will refer the reader to \cite[Chapter 6]{BeBlWo06} for more details.

\begin{definition}[Canonical extensions of Boolean algebras] (cf.\ \cite[Chapter 6, Definition 104]{BeBlWo06})
The \emph{canonical extension} of a Boolean algebra $\mathbb{B}$ is a complete Boolean algebra $\mathbb{B}^\delta$ containing $\mathbb{B}$ as a sub-Boolean algerba, and such that:
\begin{itemize}
\item[](\emph{denseness}) every element of $\mathbb{B}^\delta$ is both a meet of joins and a join of meets of elements from $\mathbb{B}$;
\item[](\emph{compactness}) for all $S,T \subseteq \mathbb{B}$ with $\bigwedge S \leq \bigvee T$ in $\mathbb{B}^\delta$, there exist some finite subsets $F \subseteq S$ and $G\subseteq T$ such that $\bigwedge F \leq \bigvee G.$\footnote{In fact, this is an equivalent formulation of the definition in \cite{BeBlWo06}.}
\end{itemize}
\end{definition}
An element $x\in\mathbb{B}^\delta$ is \emph{closed} (resp.\ \emph{open})\footnote{Notice that here the definition of closedness and openness is different from the ones in the order topology introduced by the refinement order. In the remainder of the present section, closedness and openness refer to this definition.}, if it is the meet (resp.\ join) of some subset of $\mathbb{B}$. We let $\mathsf{K}(\mathbb{B}^\delta)$ and $\mathsf{O}(\mathbb{B}^\delta)$ respectively denote the sets of closed and open elements of $\mathbb{B}^\delta$. It is easy to see that elements in $\mathbb{B}$ are exactly the ones which are both closed and open (i.e.\ \emph{clopen}).

It is well known that the canonical extension of a Boolean algebra $\mathbb{B}$ is unique up to an isomorphism for any given $\mathbb{B}$, and that assuming the axiom of choice, the canonical extension of a Boolean algebra is a perfect Boolean algebra, i.e.\ a complete and atomic Boolean algebra (cf.\ e.g.\ \cite[page 90-91]{Ho16}).

The following properties can be easily checked:
\begin{lemma}
\begin{enumerate}
\item $(\mathbb{B}^{\partial})^{\delta}\cong(\mathbb{B}^{\delta})^{\partial}$;
\item $(\mathbb{B}^{n})^{\delta}\cong(\mathbb{B}^{\delta})^{n}$;
\item $(\mathbb{B}^{\epsilon})^{\delta}\cong(\mathbb{B}^{\delta})^{\epsilon}$;
\item $\mathsf{K}((\mathbb{B}^{\partial})^{\delta})=\mathsf{O}((\mathbb{B}^{\delta})^{\partial})$;
\item $\mathsf{O}((\mathbb{B}^{\partial})^{\delta})=\mathsf{K}((\mathbb{B}^{\delta})^{\partial})$;
\item $\mathsf{K}((\mathbb{B}^{n})^{\delta})=(\mathsf{K}(\mathbb{B}^{\delta}))^{n}$;
\item $\mathsf{O}((\mathbb{B}^{n})^{\delta})=(\mathsf{O}(\mathbb{B}^{\delta}))^{n}$.
\end{enumerate}
\end{lemma}

Let $\mathbb{A},\mathbb{B}$ be Boolean algebras. An order-preserving map $f:\mathbb{A}\rightarrow \mathbb{B}$ can be extended to a map $\ca\to\B^{\delta}$ in two canonical ways. Let $f^\sigma$ and $f^\pi$ respectively denote the $\sigma$ and $\pi$\emph{-extension} of $f$ (\cite[page 375]{BeBlWo06}) defined as follows:
\begin{definition}[$\sigma$- and $\pi$-extension]\label{adef:canonical:extension:maps}
For any order-preserving map $f:\mathbb{A}\rightarrow\mathbb{B}$ and all $u\in \mathbb{A}^\delta$, we define
\[f^\sigma (u) =\bigvee \{ \bigwedge \{f(a): x\leq a\in \mathbb{A}\}: u \geq x \in \mathsf{K}(\mathbb{A}^\delta)\}\]
\[f^\pi (u) =\bigwedge \{ \bigvee \{f(a): y\geq a\in \mathbb{A}\}: u \leq y \in \mathsf{O}(\mathbb{A}^\delta)\}.\]
\end{definition}
For any map $f:\mathbb{A}\to\mathbb{B}$ we let $f^{\partial}:\mathbb{A}^{\partial}\to\mathbb{B}^{\partial}$ be such that $f^{\partial}(a):=f(a)$ for all $a\in\A$, we have that $(f^{\partial})^{\sigma}=(f^{\pi})^{\partial}$ and $(f^{\partial})^{\pi}=(f^{\sigma})^{\partial}$.

Since in a BAO, $\Box$ is meet-preserving, it is \emph{smooth}, i.e.\ $\Box^{\sigma}=\Box^{\pi}$ (cf.\ \cite[Proposition 111(3)]{BeBlWo06}). Therefore, the canonical extension of a BAO can be defined as follows:
\begin{definition}[Canonical extensions of BAOs]
The canonical extension of any BAO $(\mathbb{B}, \Box)$ is $(\mathbb{B}^{\delta}, \Box^{\sigma})= (\mathbb{B}^{\delta}, \Box^{\pi}).$
\end{definition}
\begin{definition}[Perfect Boolean algebra](cf.\ \cite[Chapter 6, Definition 40]{BeBlWo06})
A BAO $\mathbb{B}=(B,\bot,\top,\land,\lor,-,\Box)$ is said to be \emph{perfect}, if $\mathbb{B}$ is complete, atomic and completely multiplicative.
\end{definition}
As is argued in \cite[page 90-91]{Ho16}, with the axiom of choice, it can be shown that if $\mathbb{B}$ is a BAO, then $\mathbb{B}^{\delta}$ is a perfect BAO. When the axiom of choice is not available, the canonical extension of a BAO cannot be shown to be prefect in general. What can be shown is that the canonical extension of a BAO is the constructive canonical extension defined in e.g.\ \cite{GeHa01}, which is complete and completely multiplicative.

As is shown in \cite[Theorem 5.46]{Ho16}, $\mathbb{B}_{\mathsf{RO}}$ is in fact the constructive canonical extension of $\mathbb{B}_{\mathsf{FD}}$. Therefore, the following diagram shows the relation between filter-descriptive possibility frames and their underlying full possibility frames, as well as their duals:
\begin{center}
\begin{tikzpicture}[node/.style={circle, draw, fill=black}, scale=1]\label{atable:canonical:extension}
\node (BFD) at (-1.5,-1.5) {$\mathbb{B}_{\mathsf{FD}}$};
\node (BRO) at (-1.5,1.5) {$\mathbb{B}_{\mathsf{RO}}$};
\node (FD) at (1.5,-1.5) {$\mathbb{F}_{\mathsf{FD}}$};
\node (RO) at (1.5,1.5) {$\mathbb{F}_{\mathsf{RO}}$};
\draw [right hook->] (BFD) to node[left]{$(\cdot)^{\delta}$} (BRO);
\draw [->] (FD) to node[right]{$U$} (RO);
\draw [<->] (BFD) to node[above] {$\cong^{\partial}$} (FD);
\draw [<->] (BRO) to node[above] {$\cong^{\partial}$} (RO);
\end{tikzpicture}
\end{center}

Here $\cong^{\partial}$ means dual equivalence, $U$ is the forgetful functor dropping the admissible condition and replace the admissible set to the set of regular opens in the downset topology, $(\cdot)^{\delta}$ is taking the constructive canonical extension.

Using the definitions and constructions given above, it is possible to define the notion of filter-canonicity:

\begin{definition}[Filter-canonicity]
We say that an inequality $\phi\leq\psi$ is \emph{filter-canonical}, if whenever it is valid on a filter-descriptive possibility frame, it is also valid on its underlying full possibility frame.\footnote{Notice that here our definition is different from \cite[Definition 7.15]{Ho16}, which is based on the notion of canonical possibility models and frames.}
\end{definition}

By the duality theory of possibility semantics (see \cite[Section 5]{Ho16}), filter-canonicity above is equivalent to the preservation under taking constructive canonical extension.

Now we can come back to the U-shaped argument given on page \pageref{atable:U:shape:descriptive}. This argument starts from the top-left corner with the validity of the input inequality $\phi\leq\psi$ on $\mathbb{F}_{\mathsf{FD}}$, then reformulate it as the validity of the inequality in $\mathbb{B}_{\mathsf{RO}}$ with propositional variables interpreted as elements in $\mathbb{B}_{\mathsf{FD}}$, and use the algorithm ALBA to transform the inequality into an equivalent quasi-inequality Pure$(\phi\leq\psi)$ as well as its first-order translation, and then go back to the dual filter-descriptive possibility frame. Since the validity of the first-order formula does not depend on the admissible set, the bottom equivalence is obvious. The right half of the argument goes on the side of full possibility frames and their duals, the soundness of which is already shown in Section \ref{aSec:soundness}.

Indeed, the U-shaped argument on page \pageref{atable:U:shape:descriptive} gives the following results:

\begin{itemize}
\item Correspondence results with respect to filter-descriptive frames, which only uses the left arm of the U-shaped argument;
\item Canonicity results, which uses the whole U-shaped argument\footnote{In fact, as is mentioned in \cite[Section 7]{Ho16}, using techniques from \cite{CP:constructive}, the canonicity results can also be obtained via another U-shaped argument where nominals are interpreted as closed elements. Therefore, our proof here can be regarded as an alternative proof which has its relational counterpart.};
\end{itemize}

Since the equivalences of the right arm of the U-shaped argument is already shown in Section \ref{aSec:soundness}, and the bottom equivalence is obvious, we will focus on the equivalences of the left-arm, i.e.\ the soundness of the algorithm with respect to filter-descriptive possibility frames and their duals.

\subsection{Soundness over filter-descriptive possibility frames}\label{asec:soundness:topological}
In the present section we will prove the soundness of the algorithm ALBA with respect to the dual BAOs of filter-descriptive possibility frames. Indeed, similar to other semantic settings (see e.g.\ \cite{CoPa12}), the soundness proof of the filter-descriptive possibility frame side goes similar to that of the full possibility frame side (i.e.\ Theorem \ref{aCrctns:Theorem}), and for every rule except for the Ackermann rules, the proof goes without modification, thus we will only focus on the Ackermann rules here, which is justified by the topological Ackermann lemmas given below. 

\subsubsection{Topological Ackermann lemmas}\label{asubsection:topo:Ackermann}

In the present section we prove the topological Ackermann lemmas, which is the technical core of the soundness proof of the Ackermann rules with respect to filter-descriptive possibility frames. The proof is analogous to the topological Ackermann lemmas in the existing literature (e.g.\ \cite{PaSoZh15}), and we only expand on the proofs which are different.

For the Ackermann rules, the soundness proof with respect to full possibility frames is justified by the following Ackermann lemmas:

\begin{lemma}[Right-handed Ackermann lemma]\label{aRight:Ack:discrete} Let $\alpha$ be such that $p\notin\mathsf{Prop}(\alpha)$, let $\beta_1(p), \ldots, \beta_n(p)$ be positive in $p$, let $\gamma_1(p), \ldots, \gamma_n(p)$ be negative in $p$, and let $\vec q, \vec \nomj$ be all the propositional variables, nominals, respectively, occurring in $\alpha, \beta_1(p), \ldots, \beta_n(p)$, $\gamma_1(p), \ldots, \gamma_n(p)$ other than $p$. Then for all $\vec a\in\mathbb{B}_{\mathsf{RO}}, \vec x\in\mathsf{PsAt}(\mathbb{B}_{\mathsf{RO}})$, the following are equivalent:
\begin{enumerate}
\item
$\beta_i^{\mathbb{B}_{\mathsf{RO}}}(\vec a, \vec x, \alpha^{\mathbb{B}_{\mathsf{RO}}}(\vec a, \vec x))\leq\gamma_i^{\mathbb{B}_{\mathsf{RO}}}(\vec a, \vec x, \alpha^{\mathbb{B}_{\mathsf{RO}}}(\vec a, \vec x))$ for $1\leq i\leq n$;
\item
There exists $a_0\in\mathbb{B}_{\mathsf{RO}}$ such that $\alpha^{\mathbb{B}_{\mathsf{RO}}}(\vec a, \vec x)\leq a_0$ and $\beta_i^{\mathbb{B}_{\mathsf{RO}}}(\vec a, \vec x, a_0)\leq\gamma_i^{\mathbb{B}_{\mathsf{RO}}}(\vec a, \vec x, a_0)$ for $1\leq i\leq n$.
\end{enumerate}
\end{lemma}

\begin{lemma}[Left-handed Ackermann lemma]\label{aLeft:Ack:discrete}
Let $\alpha$ be such that $p\notin \mathsf{Prop}(\alpha)$, let $\beta_1(p), \ldots, \beta_n(p)$ be negative in $p$, let $\gamma_1(p), \ldots, \gamma_n(p)$ be positive in $p$, and let $\vec q, \vec \nomj$ be all the propositional variables, nominals, respectively, occurring in $\alpha, \beta_1(p), \ldots, \beta_n(p)$, $\gamma_1(p), \ldots, \gamma_n(p)$ other than $p$. Then for all $\vec a\in\mathbb{B}_{\mathsf{RO}}, \vec x\in \mathsf{PsAt}(\mathbb{B}_{\mathsf{RO}})$, the following are equivalent:
\begin{enumerate}
\item
$\beta_i^{\mathbb{B}_{\mathsf{RO}}}(\vec a, \vec x, \alpha^{\mathbb{B}_{\mathsf{RO}}}(\vec a, \vec x))\leq\gamma_i^{\mathbb{B}_{\mathsf{RO}}}(\vec a, \vec x, \alpha^{\mathbb{B}_{\mathsf{RO}}}(\vec a, \vec x))$ for $1\leq i\leq n$;
\item
There exists $a_0\in\mathbb{B}_{\mathsf{RO}}$ such that $a_0\leq\alpha^{\mathbb{B}_{\mathsf{RO}}}(\vec a, \vec x)$ and $\beta_i^{\mathbb{B}_{\mathsf{RO}}}(\vec a, \vec x, a_0)\leq\gamma_i^{\mathbb{B}_{\mathsf{RO}}}(\vec a, \vec x, a_0)$ for $1\leq i\leq n$.
\end{enumerate}
\end{lemma}

As is similar to what is discussed in the existing literature (e.g.\ \cite[Section 9]{CoPa12}), the soundness proof of the Ackermann rules, namely the Ackermann lemmas, cannot be straightforwardly adapted to the setting of admissible assignments, since formulas in the expanded modal language $\mathcal{L}^{+}$ might be interpreted as elements in $\mathbb{B}_{\mathsf{RO}}\setminus\mathbb{B}_{\mathsf{FD}}$ even if all the propositional variables are interpreted in $\mathbb{B}_{\mathsf{FD}}$, thus we cannot find the $a_0\in\mathbb{B}_{\mathsf{FD}}$ required in the setting of admissible assignments. Therefore, some adaptations are necessary based on the syntactic shape of the formulas, the definitions of which are analogous to \cite[Definition B.3]{PaSoZh15}:

\begin{definition}[Syntactically closed and open formulas]\label{adef:closed:open}
\begin{enumerate}
\item A formula in $\mathcal{L}^{+}$ is \emph{syntactically closed} if all occurrences of nominals and $\Diamondblack$ are positive, and all occurrences of $\Boxblack$ are negative;
\item A formula in $\mathcal{L}^{+}$ is \emph{syntactically open} if all occurrences of nominals and $\Diamondblack$ are negative, and all occurrences of $\Boxblack$ are positive.
\end{enumerate}
\end{definition}

As is discussed in \cite[Section 9]{CoPa12}, the intuition behind these definitions is that the value of a syntactically open (resp.\ closed) formula under an admissible assignment is always an open (resp.\ closed) element in $\mathbb{B}_{\mathsf{RO}}$, i.e., in $\mathsf{O}(\mathbb{B}_{\mathsf{RO}})$ (resp.\ $\mathsf{K}(\mathbb{B}_{\mathsf{RO}})$), therefore we can apply compactness to obtain an admissible $a_0$ as required by the Ackermann lemmas stated below, which are analogous to \cite[Lemma B.4, B.5]{PaSoZh15}:

\begin{lemma}[Right-handed topological Ackermann lemma]\label{aRight:Ack} Let $\alpha$ be syntactically closed, $p\notin\mathsf{Prop}(\alpha)$, let $\beta_1(p), \ldots, \beta_n(p)$ be syntactically closed and positive in $p$, let $\gamma_1(p), \ldots, \gamma_n(p)$ be syntactically open and negative in $p$, and let $\vec q, \vec \nomj$ be all the propositional variables, nominals, respectively, occurring in $\alpha, \beta_1(p), \ldots, \beta_n(p)$, $\gamma_1(p), \ldots, \gamma_n(p)$ other than $p$. Then for all $\vec a\in\mathbb{B}_{\mathsf{FD}}, \vec x\in\mathsf{PsAt}(\mathbb{B}_{\mathsf{RO}})$, the following are equivalent:
\begin{enumerate}
\item
$\beta_i^{\mathbb{B}_{\mathsf{RO}}}(\vec a, \vec x, \alpha^{\mathbb{B}_{\mathsf{RO}}}(\vec a, \vec x))\leq\gamma_i^{\mathbb{B}_{\mathsf{RO}}}(\vec a, \vec x, \alpha^{\mathbb{B}_{\mathsf{RO}}}(\vec a, \vec x))$ for $1\leq i\leq n$;
\item
There exists $a_0\in\mathbb{B}_{\mathsf{FD}}$ such that $\alpha^{\mathbb{B}_{\mathsf{RO}}}(\vec a, \vec x)\leq a_0$ and $\beta_i^{\mathbb{B}_{\mathsf{RO}}}(\vec a, \vec x, a_0)\leq\gamma_i^{\mathbb{B}_{\mathsf{RO}}}(\vec a, \vec x, a_0)$ for $1\leq i\leq n$.
\end{enumerate}
\end{lemma}

\begin{lemma}[Left-handed topological Ackermann lemma]\label{aLeft:Ack}

Let $\alpha$ be syntactically open, $p\notin \mathsf{Prop}(\alpha)$, let $\beta_1(p), \ldots, \beta_n(p)$ be syntactically closed and negative in $p$, let $\gamma_1(p), \ldots, \gamma_n(p)$ be syntactically open and positive in $p$, and let $\vec q, \vec \nomj$ be all the propositional variables, nominals, respectively, occurring in $\alpha, \beta_1(p), \ldots, \beta_n(p)$, $\gamma_1(p), \ldots, \gamma_n(p)$ other than $p$. Then for all $\vec a\in\mathbb{B}_{\mathsf{FD}}, \vec x\in \mathsf{PsAt}(\mathbb{B}_{\mathsf{RO}})$, the following are equivalent:
\begin{enumerate}
\item
$\beta_i^{\mathbb{B}_{\mathsf{RO}}}(\vec a, \vec x, \alpha^{\mathbb{B}_{\mathsf{RO}}}(\vec a, \vec x))\leq\gamma_i^{\mathbb{B}_{\mathsf{RO}}}(\vec a, \vec x, \alpha^{\mathbb{B}_{\mathsf{RO}}}(\vec a, \vec x))$ for $1\leq i\leq n$;
\item
There exists $a_0\in\mathbb{B}_{\mathsf{FD}}$ such that $a_0\leq\alpha^{\mathbb{B}_{\mathsf{RO}}}(\vec a, \vec x)$ and $\beta_i^{\mathbb{B}_{\mathsf{RO}}}(\vec a, \vec x, a_0)\leq\gamma_i^{\mathbb{B}_{\mathsf{RO}}}(\vec a, \vec x, a_0)$ for $1\leq i\leq n$.
\end{enumerate}
\end{lemma}

In order to prove the topological Ackermann lemmas stated above, we need to prove some properties of the connectives and formulas, which are analogous to their counterparts in \cite[Section B]{PaSoZh15}. The proofs of the items are also similar, so we only provide the proofs for the items which have different proofs.

Indeed, the next lemma is the only one which needs a different proof, since elements in $\mathsf{PsAt}(\mathbb{B}_{\mathsf{RO}})$ are not necessarily complete join-irreducible elements.
\begin{lemma}
$\mathsf{PsAt}(\mathbb{B}_{\mathsf{RO}})\subseteq\mathsf{K}(\mathbb{B}_{\mathsf{RO}})$.
\end{lemma}

\begin{proof}
It suffices to show that for any element $w\in W$, $\mathsf{ro}(\{w\})=\bigwedge\{a\in\mathbb{B}_{\mathsf{FD}}\mid w\in a\}$. It is easy to see that $\mathsf{ro}(\{w\})\leq\bigwedge\{a\in\mathbb{B}_{\mathsf{FD}}\mid w\in a\}$. Suppose that the inequality is strict, then there exists a $v\in W$ such that $v\notin\mathsf{ro}(\{w\})$ and $v\in\bigwedge\{a\in\mathbb{B}_{\mathsf{FD}}\mid w\in a\}$. By Definition \ref{adef:filter:desctiptive:frame}, every filter-descriptive possibility frame is $\sqsubseteq$-tight, therefore $v\sqsubseteq w$. Since $\mathsf{ro}(\{w\})$ is downward closed, we have $v\in \mathsf{ro}(\{w\})$, a contradiction.
\end{proof}

\begin{lemma}
\label{acor:congenial top for whites and blacks} For all $c\in\mathsf{K}(\mathbb{B}_{\mathsf{RO}})$ and $o\in \mathsf{O}(\mathbb{B}_{\mathsf{RO}})$,

\begin{enumerate}
\item $\Box_{\mathsf{RO}}(c), \Diamond_{\mathsf{RO}}(c), \Diamondblack_{\mathsf{RO}}(c)\in\mathsf{K}(\mathbb{B}_{\mathsf{RO}})$;
\item $\Box_{\mathsf{RO}}(o), \Diamond_{\mathsf{RO}}(o), \Boxblack_{\mathsf{RO}}(o)\in\mathsf{O}(\mathbb{B}_{\mathsf{RO}})$.
\end{enumerate}
\end{lemma}

\begin{lemma}\label{aUpward-Directed}

Let $\mathcal{C}=\{c_i:i\in I\}\subseteq\mathsf{K}(\mathbb{B}_{\mathsf{RO}})$ be non-empty and downward-directed, $\mathcal{O}=\{o_i:i\in I\}\subseteq\mathsf{O}(\mathbb{B}_{\mathsf{RO}})$ be non-empty and upward-directed, then
\begin{enumerate}
\item $\Box_{\mathsf{RO}}(\bigvee\mathcal{O})=\bigvee\{\Box_{\mathsf{RO}}(o): o\in\mathcal{O}\}$;
\item $\Diamond_{\mathsf{RO}}(\bigwedge\mathcal{C})=\bigwedge\{\Diamond_{\mathsf{RO}}(c): c\in\mathcal{C}\}$;
\item $\Boxblack_{\mathsf{RO}}(\bigvee\mathcal{O})=\bigvee\{\Boxblack_{\mathsf{RO}}(o): o\in\mathcal{O}\}$;
\item $\Diamondblack_{\mathsf{RO}}(\bigwedge\mathcal{C})=\bigwedge\{\Diamondblack_{\mathsf{RO}}(c): c\in\mathcal{C}\}$.
\end{enumerate}
\end{lemma}

\begin{lemma}\label{alemma:formula:close:open}
Let $\phi(p)$ be syntactically closed and $\psi(p)$ syntactically open, and let $\vec p, \vec \nomj$ be all the propositional variables, nominals, respectively, occuring in $\phi(p)$ and $\psi(p)$ other than $p$, $\vec a\in\mathbb{B}_{\mathsf{FD}}, \vec x\in\mathsf{PsAt}(\mathbb{B}_{\mathsf{RO}}), c\in\mathsf{K}(\mathbb{B}_{\mathsf{RO}}), o\in\mathsf{O}(\mathbb{B}_{\mathsf{RO}})$, then
\begin{enumerate}
\item If $\phi(p)$ is positive in $p$, then $\phi^{\mathbb{B}_{\mathsf{RO}}}(\vec a,\vec x,c)\in\mathsf{K}(\mathbb{B}_{\mathsf{RO}})$;
\item If $\psi(p)$ is negative in $p$, then $\psi^{\mathbb{B}_{\mathsf{RO}}}(\vec a,\vec x,c)\in\mathsf{O}(\mathbb{B}_{\mathsf{RO}})$;
\item If $\phi(p)$ is negative in $p$, then $\phi^{\mathbb{B}_{\mathsf{RO}}}(\vec a,\vec x,o)\in\mathsf{K}(\mathbb{B}_{\mathsf{RO}})$;
\item If $\psi(p)$ is positive in $p$, then $\psi^{\mathbb{B}_{\mathsf{RO}}}(\vec a,\vec x,o)\in\mathsf{O}(\mathbb{B}_{\mathsf{RO}})$.
\end{enumerate}
\end{lemma}

\begin{lemma}\label{aMJ:Pres}
Let $\phi(p)$ be syntactically closed, $\psi(p)$ be syntactically open, and let $\vec p, \vec \nomj$ be all the proposition variables, nominals, respectively, occuring in $\phi(p)$ and $\psi(p)$ other than $p$, $\vec a\in\mathbb{B}_{\mathsf{FD}}, \vec x\in\mathsf{PsAt}(\mathbb{B}_{\mathsf{RO}})$, $\{c_i : i\in I\}\subseteq \mathsf{K}(\mathbb{B}_{\mathsf{RO}})$ be non-empty and downward-directed, $\{o_i : i\in I\}\subseteq \mathsf{O}(\mathbb{B}_{\mathsf{RO}})$ be non-empty and upward-directed, then
\begin{enumerate}
\item If $\phi(p)$ is positive in $p$, then $\phi^{\mathbb{B}_{\mathsf{RO}}}(\vec a, \vec x, \bigwedge\{c_i : i\in I\})=\bigwedge\{\phi^{\mathbb{B}_{\mathsf{RO}}}(\vec a, \vec x, c_i): i\in I\}$;
\item If $\psi(p)$ is negative in $p$, then $\psi^{\mathbb{B}_{\mathsf{RO}}}(\vec a, \vec x, \bigwedge\{c_i : i\in I\})=\bigvee\{\psi^{\mathbb{B}_{\mathsf{RO}}}(\vec a, \vec x, c_i): i\in I\}$;
\item If $\phi(p)$ is negative in $p$, then $\phi^{\mathbb{B}_{\mathsf{RO}}}(\vec a, \vec x, \bigvee\{o_i : i\in I\})=\bigwedge\{\phi^{\mathbb{B}_{\mathsf{RO}}}(\vec a, \vec x, o_i): i\in I\}$;
\item If $\psi(p)$ is positive in $p$, then $\psi^{\mathbb{B}_{\mathsf{RO}}}(\vec a, \vec x, \bigvee\{o_i : i\in I\})=\bigvee\{\psi^{\mathbb{B}_{\mathsf{RO}}}(\vec a, \vec x, o_i): i\in I\}$.
\end{enumerate}
\end{lemma}
\begin{cor}\label{aClose:to:Close}

Let $\phi$ be syntactically closed and $\psi$ syntactically open, and let $\vec p, \vec \nomj$ be all the proposition variables, nominals, respectively, occurring in $\phi$ and $\psi$, $\vec a\in\mathbb{B}_{\mathsf{FD}}, \vec x\in\mathsf{PsAt}(\mathbb{B}_{\mathsf{RO}})$. Then

\begin{enumerate}
\item $\phi^{\bbas}(\vec a, \vec x, \vec y)\in\mathsf{K}(\mathbb{B}_{\mathsf{RO}})$;
\item $\psi^{\bbas}(\vec a, \vec x, \vec y)\in\mathsf{O}(\mathbb{B}_{\mathsf{RO}})$.
\end{enumerate}
\end{cor}

\subsection{Results}\label{aSec:results:topo}

Using the soundness of the algorithm ALBA with respect to filter-descriptive possibility frames, we have the following results:

\begin{theorem}
For any inequality $\phi\leq\psi$ in the basic modal language $\mathcal{L}$ on which ALBA succeeds and outputs $\mathsf{FO}(\phi\leq\psi)$, $\phi\leq\psi$ corresponds to $\mathsf{FO}(\phi\leq\psi)$ in the sense that for any filter-descriptive possibility frame $\mathbb{F}$, $\mathbb{F}\vDash\phi\leq\psi$ if and only if $\mathbb{F}\vDash\mathsf{FO}(\phi\leq\psi)$.
\end{theorem}

\begin{proof}
By the left-arm of the U-shaped argument given on page \pageref{atable:U:shape:descriptive}.
\end{proof}

\begin{theorem}
Any inequality $\phi\leq\psi$ in the basic modal language $\mathcal{L}$ on which ALBA succeeds is preserved under taking constructive canonical extension, and hence filter-canonical.
\end{theorem}

\begin{proof}
By the whole U-shaped argument given on page \pageref{atable:U:shape:descriptive}.
\end{proof}

\section{Conclusions and future directions}\label{aSec:Discussion}

In the present section, we will discuss some aspects of correspondence and canonicity in possibility semantics, and give some future directions.

\subsection{The variation of interpretations: nominals and approximation rules in different ALBAs}

The SQEMA\footnote{SQEMA stands for \textbf{S}econd-order \textbf{Q}uantifier \textbf{E}limination for \textbf{M}odal formulae using \textbf{A}ckermann's lemma (see \cite{CoGoVa06}).}-ALBA line of algorithmic correspondence research starts from Boolean algebra based modal logics \cite{CoGoVa06}, and later on the underlying semantic environment generalizes to distributive lattices \cite{CoPa12}, general lattices \cite{CoPa:non-dist}, constructive extensions of lattices \cite{CP:constructive} and possibility semantics. Along the line of generalizations, properties specific to certain settings are separated from the more general properties. Here we are going to discuss the rules of nominals, i.e. the approximation rules in detail.

\paragraph{Boolean algebras with operators.}In the setting of Boolean algebras with operators, every $\mathcal{CAV}$-BAO has atoms, which are complete join-irreducibles, complete join-primes and join-generators. Therefore, the following kind of approximation rule is available (see \cite{CoGoVa06}):
\begin{prooftree}
\AxiomC{$\nomi\leq\Diamond\gamma$}
\RightLabel{(Left-$\Diamond$-Appr)}
\UnaryInfC{$\nomi\leq\Diamond\nomj$\ \ $\nomj\leq\gamma$}
\end{prooftree}
where the nominals are interpreted as atoms. To guarantee the soundness of the rule above, the following properties are used (in the remainder of this subsection, we will abuse notation and identify formulas with their interpretations on algebras):
\begin{itemize}
\item First of all, $\Diamond\gamma$ is represented as $\Diamond\bigvee\{\nomj\in\mathsf{At}(\mathbb{B})\mid\nomj\leq\gamma\}$, which uses the fact that the atoms in the $\mathcal{CAV}$-BAOs are join-generators;
\item secondly, since $\Diamond$ is completely join-preserving, it can be equivalently represented as $\bigvee\{\Diamond\nomj\mid\nomj\in\mathsf{At}(\mathbb{B})\mbox{ and }\nomj\leq\gamma\}$;
\item finally, $\nomi\leq\bigvee\{\Diamond\nomj\mid\nomj\in\mathsf{At}(\mathbb{B})\mbox{ and }\nomj\leq\gamma\}$ iff there exists a $\nomj\in\mathsf{At}(\mathbb{B})$ such that $\nomj\leq\gamma$ and $\nomi\leq\Diamond\nomj$, which uses the fact that the atoms are completely join-prime.
\end{itemize}
Therefore, the semantic properties used in the soundness proof of (Left-$\Diamond$-Appr) are the following:

\begin{itemize}
\item Atoms are join-generators;
\item $\Diamond$ is completely join-preserving;
\item Atoms are completely join-prime.
\end{itemize}

Therefore, the atomicity of $\mathcal{CAV}$-BAOs are not essentially used in the soundness proof of the rule above.

\paragraph{Distributive lattice expansions.}Going from Boolean algebras with operators to distributive lattice expansions, the atomicity is not available anymore. However, in perfect distributive lattices, there are complete join-irreducibles which are also complete join-primes and join-generators, although not necessarily atomic. Therefore, as we analyzed in the Boolean setting, if we interpret the nominals as complete join-irreducibles (i.e.\ complete join-primes), these properties are enough to guarantee the rule (Left-$\Diamond$-Appr) to be sound on perfect distributive lattice expansions.

\paragraph{General lattice expansions.}In the further general setting of perfect general lattice expansions without assuming distributivity, complete join-irreducibles are not the same as complete join-primes anymore. The remaining property in general lattice expansions is that the complete join-irreducibles (here we denote the set of complete join-irreducibles of the lattice $\mathbb{L}$ as $\mathsf{J}(\mathbb{L})$) are still join-generators. Consider the following rule:
\begin{center}
\AxiomC{$(S, \;\; \phi(\gamma / !x)\leq \psi)$}
\RightLabel{$(L^+A)$}
\UnaryInfC{$(S\! \cup\! \{ \nomi \leq \gamma\},\;\; \phi(\nomi / !x)\leq \psi)$}
\DisplayProof
\end{center}
with $+x \prec +\phi(!x)$, the branch of $+\phi(!x)$ starting at $+x$ being SLR, $\gamma$ belonging to the original modal language and $\nomi$ being the first nominal variable not occurring in $S$ or $\phi(\gamma / !x)\leq \psi$.

To guarantee the soundness of the rule above, the following properties are used:
\begin{itemize}
\item First of all, $\phi(\gamma)$ is represented as $\phi(\bigvee\{\nomj\in\mathsf{J}(\mathbb{L})\mid\nomj\leq\gamma\})$, which uses the fact that the complete join-irreducibles in prefect general lattices are join-generators;
\item secondly, since $\phi(!x)$ is completely join-preserving, it can be equivalently represented as $\bigvee\{\phi(\nomj)\mid\nomj\in\mathsf{J}(\mathbb{L})\mbox{ and }\nomj\leq\gamma\}$;
\item finally, $\bigvee\{\phi(\nomj)\mid\nomj\in\mathsf{J}(\mathbb{L})\mbox{ and }\nomj\leq\gamma\}\leq\psi$ iff for all $\nomj\in\mathsf{J}(\mathbb{L})\mbox{ s.t.\ }\nomj\leq\gamma$, it holds that $\phi(\nomj)\leq\psi$, which does not use special properties of perfect general lattice expansions.
\end{itemize}

Therefore, in an approximation rule like $(L^+A)$, the semantic property essentially used is that the complete join-irreducibles are join-generators. As a result, any complete lattice-like structures which have join-generators can have an approximation rule in the style of $(L^+A)$ to be sound, once the nominals are interpreted as the join-generators.

For example, in the constructive canonical extensions of general lattice expansions, they are not necessarily perfect, so there are not ``enough'' complete join-irreducibles. However, in a canonical extension, every element can be represented as the join of closed elements, therefore if we interpret the nominals as closed elements, the approximation rule above is sound.

In possibility semantics, the regular open dual BAOs of full possibility frames are $\mathcal{CV}$-BAOs, thus do not have ``enough'' atoms. However, in this setting, the regular open closures of singletons can serve as the join-generators, which gives the semantic environment for the rule $(L^+A)$ to be sound.

To sum up, the following table shows the semantic properties available in each setting, which justifies the use of different approximation rules in each setting:

\begin{center}
\begin{tabular}{|c|c|c|}

\hline
\textbf{Propositional base}&\textbf{Nominals/join-generators}&\textbf{Complete join-primes}\\
\hline
Boolean algebras&atoms&atoms\\
\hline
distributive lattices&complete join-irreducibles&complete join-irreducibles\\
\hline
general lattices&complete join-irreducibles&not enough\\
\hline
constructive canonical extensions&closed elements&not enough\\
\hline
possibility semantics&regular open closures of singletons&not enough\\
\hline
\end{tabular}
\end{center}

\subsection{The essence of minimal valuation}

As we can see from the previous subsection, nominals are not necessarily interpreted as atoms (singletons) or complete join-irreducibles, anything that can be taken as the join-generators can serve as the interpretation. Therefore, the execution of the algorithm ALBA provides quasi-inequalities equivalent to the input formula (inequality), which are in the language with nominals and the connectives in the expanded modal language. Therefore, a successful ALBA reduction of a modal formula to the pure language means that the input modal formula can be expressed by the join-generators. In the setting of constructive canonical extensions of general lattices, there is not (yet) an obvious way of expressing closed elements in the first-order correspondence language. However, in other settings, we can express the atoms, complete join-irreducibles or regular open closures of singletons in the first-order correspondence language. Since we also have the first-order translation of the connectives both in the basic modal language and in the expanded modal language, we can thus translate the pure quasi-inequalities into first-order formulas.

\subsection{Translation method and its limitations}

Since possibility frames have two binary relations, it is natural to view them as bimodal Kripke frames with additional restrictions on the valuations of propositional variables (see \cite{vBBeHo16} for a detailed discussion of this bimodal perspective). Therefore, a natural question is whether we can use this view to reduce correspondence problems in possibility semantics to correspondence problems in the bimodal language, like using the G\"{o}del translation from intuitionistic logic to S4 modal logic (see \cite{CPZ:Trans} for a detailed discussion of the power and limits of translation method in correspondence theory). As we can see in e.g.\ \cite{vBBeHo16}, when we try to reduce correspondence problems in possibility semantics to correspondence problems in the bimodal language, a problem arises due to the complication of the formula structure of the bimodal formulas after the translation. However, since there are additional properties satisfied in the bimodal language (e.g.\ the interaction axioms between the accessibility relation and the refinement relation), there are additional order-theoretic properties satisfied, which can make more connective combinations having the nice order-theoretic properties. One example is $\Box_\sqsubseteq\Diamond_\sqsubseteq$, which is meet-preserving. If we make the analysis from the prespective of the order-theoretic properties of the individual connectives only, then $\Box_\sqsubseteq\Diamond_\sqsubseteq$ is of the shape of the antecedent of the McKinsey formula, which has a very bad combination pattern.

Therefore, a natural question is: can we use the order-theoretic properties of the combinations of connectives in correspondence theory for superintuitionistic logics, in addition to the properties of the individual connectives, to obtain results like in \cite{Ro86}?

\subsection{Constructive canonical extensions}
As we can see in \cite[Theorem 5.46]{Ho16} and the canonicity proof in the present paper, the persistence notion corresponding to the validity preservation from filter-descriptive possibility frames to their underlying full possibility frames is exactly the same as the validity preservation under taking constructive canonical extension. Therefore, possibility semantics provides a relational semantic environment to give canonicity proofs without appealing to the axiom of choice and its equivalent forms. In addition, possibility frames could be recognized as the dual relational semantic environment to study constructive canonical extensions \cite{GeHa01,GhMe97}. In this sense, we can say that possibility semantics is the constructive counterpart of Kripke semantics.

A future question related to the constructive feature is about the frame-theoretic counterpart of constructive canonical extensions in lattice-based settings, i.e.\ non-classical versions of possibility semantics as the duals of distributive lattices and their canonical extensions, which is expected to ``relationalize'' the constructive canonical extensions in a ``pointed'' way, where ``points'' refer to filters rather than prime filters or ultrafilters. For the intuitionistic generalization of possibility semantics, it is already studied in \cite{BeHo16} as Dragalin semantics.

\bibliographystyle{abbrv}
\bibliography{Correspondence_ALBA}

\end{document}